\documentclass[11pt,a4paper]{amsart}
\usepackage[latin1]{inputenc}
\usepackage{amsmath}
\usepackage{amsfonts}
\usepackage{amssymb}
\usepackage{amsthm}
\usepackage{anysize}
\usepackage{enumitem}
\usepackage{amscd}
\usepackage{hyperref}
\usepackage[square, comma, numbers, sort&compress]{natbib}
\usepackage{indentfirst}
\marginsize{3.0cm}{3.0cm}{3.0cm}{3.0cm}
\setenumerate{label={\normalfont(\arabic*)}}

\newcommand{\form}[1]{{\langle #1 \rangle }}
\newcommand{\pfister}[1]{{\langle \! \langle #1 \rangle \! \rangle}}

\newcommand{\witti}[2]{{\mathfrak{i}_{#1}(#2)}}
\newcommand{\wittj}[2]{{\mathfrak{j}_{#1}(#2)}}
\newcommand{\anispart}[1]{{#1_{\mathrm{an}}}}

\newcommand{\normform}[1]{{#1_{\mathrm{nor}}}}

\newtheorem{theorem}{Theorem}[section]

\newtheorem{lemma}[theorem]{Lemma}
\newtheorem{sublemma}[theorem]{Sublemma}
\newtheorem{proposition}[theorem]{Proposition}
\newtheorem{corollary}[theorem]{Corollary}

\theoremstyle{definition}

\newtheorem{example}[theorem]{Example}

\theoremstyle{remark}
\newtheorem{remark}[theorem]{Remark}

\numberwithin{equation}{section}
\begin{document}

\title[]{A bound for the index of a quadratic form after scalar extension to the function field of a quadric}
\author{Stephen Scully}
\address{Department of Mathematical and Statistical Sciences, University of Alberta, Edmonton AB T6G 2G1, Canada}
\email{stephenjscully@gmail.com}

\subjclass[2010]{11E04, 14E05, 15A03.}
\keywords{Quadratic forms, isotropy indices, quadrics and quadratic Grassmannians, rational maps between quadratic Grassmannians}

\begin{abstract} Let $q$ be an anisotropic quadratic form defined over a general field $F$. In this article, we formulate a new upper bound for the isotropy index of $q$ after scalar extension to the function field of an arbitrary quadric. On the one hand, this bound offers a refinement of a celebrated bound established in earlier work of Karpenko-Merkurjev and Totaro; on the other, it is a direct generalization of Karpenko's theorem on the possible values of the first higher isotropy index. We prove its validity in two important cases: (i) the case where $\mathrm{char}(F) \neq 2$, and (ii) the case where $\mathrm{char}(F) = 2$ and $q$ is \emph{quasilinear} (i.e., diagonalizable). The two cases are treated separately using completely different approaches, the first being algebraic-geometric, and the second being purely algebraic. \end{abstract}

\maketitle

\section{Introduction} \label{Introduction} A basic, yet fundamental tool in the study of quadratic forms over general fields is that of scalar extension. Among its many uses, perhaps the most significant is that of forcing some quadratic form of interest to acquire a non-trivial zero, assuming that none exist in the base field (that is, forcing an \emph{anisotropic} form to become \emph{isotropic}). This has the effect of lowering the ``anisotropic dimension'' of the form, which, apart from providing a natural means by which to argue inductively, can often lead one to witness non-trivial behaviour in other quadratic forms of interest as the field of definition is enlarged. In order to draw meaningful conclusions, the challenge is typically then one of determining restrictions of a general nature on this kind of behaviour. What those restrictions are, however, will depend heavily on the particular choice of extension field. In order to maintain as much control as possible, it is often desirable to choose an extension for which the restrictions are most severe. Philosophically, this means that one should choose an extension which, in the sense of valuation theory, is ``generic'' for the problem at hand. For the purpose of forcing an anisotropic quadratic form to become isotropic, there is a canonical choice of extension having this property; namely, the function field of the projective quadric defined by the vanishing of that form. For these and other reasons, studying the effect of scalar extension to function fields of quadrics has been a dominant theme in much of the research carried out in the algebraic theory of quadratic forms since the early 1970s. 

One problem of central interest in this area is the following: Let $p$ and $q$ be anisotropic quadratic forms of dimension $\geq 2$ over a general field $F$, and let $P$ and $Q$ denote the projective $F$-quadrics defined by the vanishing of $p$ and $q$, respectively. Under what circumstances does $q$ become isotropic over the function field $F(p)$ of the quadric $P$? From an algebraic-geometric perspective, this simply amounts to asking for necessary and sufficient conditions in order for there to exist a rational map $P \dashrightarrow Q$ over $F$. Nevertheless, the fact that we impose no constraints on the triple $(F, q,p)$ endows the problem with a depth and complexity which belies its initial appearance. In fact, it is entirely unreasonable to hope for a complete solution in this generality; as a vast literature accumulated over the past three decades amply demonstrates, the problem is already considerably involved in low-dimensional situations (see, e.g., \cite[\S 8.2]{Kahn}).

In the present article, we will thus be concerned with a weaker variant of this problem which seeks to investigate restrictions of a particular kind on the isotropy index $\witti{0}{q_{F(p)}}$ of $q$ after scalar extension to the field $F(p)$. Recall here that if $\phi$ is a quadratic form defined on a finite-dimensional vectors space $V$ over a field, then its \emph{isotropy index} $\witti{0}{\phi}$ is defined as the maximal dimension of a subspace of $V$ on which $q$ is uniformly zero (that is, a \emph{totally isotropic subspace} of $V$). As such, the integer $\witti{0}{q_{F(p)}}$ can not only detect the isotropy of $q$ over $F(p)$, but measure the extent to which it persists; in the language of algebraic geometry, our problem is not only concerned with identifying necessary conditions in order for there to exist a rational map from $P$ to $Q$, but from $P$ to the variety of totally isotropic subspaces of any prescribed dimension in $Q$.

The restrictions which we are interested in here are of a very general nature; namely, those imposed by only the most basic discrete invariants of $p$ and $q$. In this respect, our problem has a long and rich history dating back to the late 1960s and the classic ``subform theorem'' of Cassels and Pfister. While many important contributions have been made over the years by a large number of authors, one of the most notable results established to date in this direction is a penetrating upper bound for the integer $\witti{0}{q_{F(p)}}$ involving just three invariants of $p$ and $q$: the dimension of $q$, the dimension of $p$, and the \emph{first higher isotropy index} of $p$, defined as the integer $\witti{1}{p} := \witti{0}{p_{F(p)}}$. Originally conjectured by Izhboldin (\cite{Izhboldin1}), and having its origin in the seminal work of Hoffmann (\cite{Hoffmann1}), the bound was first established in the case where $\mathrm{char}(F) \neq 2$ by Karpenko and Merkurjev (\cite[Cor. 4.2]{KarpenkoMerkurjev}), and later in full generality by Totaro (see \cite[Thm. 5.2]{Totaro1}): 

\begin{theorem}[{Karpenko-Merkurjev, Totaro}] \label{thmKMT} Let $p$ and $q$ be anisotropic quadratic forms of dimension $\geq 2$ over a field $F$. If $q_{F(p)}$ is isotropic, then
$$\witti{0}{q_{F(p)}} \leq \mathrm{dim}(q) - \mathrm{dim}(p) + \witti{1}{p}.$$ \end{theorem}

Unlike much of the work which came before it, Theorem \ref{thmKMT} was established using methods of algebraic geometry. Indeed, both sets of authors exploited the aforementioned geometric interpretation of the integer $\witti{0}{q_{F(p)}}$ in order to bring the theory of algebraic cycles to bear on the problem. This built upon foundational ideas of Vishik, who had earlier obtained non-trivial results in the same direction (\cite{Vishik2}). The theorem is all the more significant due to the fact that a considerable amount is known concerning the integer $\witti{1}{p}$. In fact, there is a precise conjectural description of its possible values due to Hoffmann, and this conjecture has been verified in the case where $\mathrm{char}(F) \neq 2$ by Karpenko (\cite{Karpenko1}). Unfortunately, Karpenko's proof makes essential use of a powerful fact which is not yet available in characteristic 2: the existence of Steenrod-type operations on the Chow groups of smooth projective varieties with $\mathbb{Z}/2\mathbb{Z}$ coefficients. Hoffmann's conjecture therefore remains open in the latter setting. In recent work of the author, however, completely different (and purely algebraic) methods were used to settle the conjecture for a special class of forms in characteristic 2. More specifically, it was shown in \cite[Thm 1.3]{Scully3} that if $\mathrm{char}(F) = 2$, then Hoffmann's conjecture on $\witti{1}{p}$ holds in the case where $p$ is \emph{quasilinear}, meaning that $p$ is isometric to a form of Fermat type, i.e., a weighted sum of squares $a_1X_1^2 + a_2X_2^2 + \cdots + a_nX_n^2$.  We are left with the following general result: 

\begin{theorem}[Karpenko, Scully] \label{thmKarpenko} Let $p$ be an anisotropic quadratic form of dimension $\geq 2$ over a field $F$, and let $s = v_2\big(\mathrm{dim}(p) - \witti{1}{p}\big)$. Assume that either
\begin{enumerate} \item The characteristic of $F$ is not $2$, or
\item The characteristic of $F$ is 2 and $p$ is quasilinear. \end{enumerate}
Then $\witti{1}{p} \leq 2^s$. \end{theorem}

Here the notation $v_2(n)$ stands for the $2$-adic order of the integer $n$. It is worth remarking that while Hoffmann's conjecture is essentially wide open for non-quasilinear forms in characteristic 2, non-trivial partial results have been obtained by Hoffmann-Laghribi (\cite{HoffmannLaghribi2}) and also by Haution as a by-product of his efforts to develop the geometric machinery which is currently absent from the characteristic-2 setting (see \cite{Haution1},\cite{Haution2}).

Theorems \ref{thmKMT} and \ref{thmKarpenko} represent important landmarks for the theory of quadratic forms. Put together, they lead to the (highly applicable) conclusion that $\witti{0}{q_{F(p)}}$ is typically not much more than $\mathrm{dim}(q) - \mathrm{dim}(p)$. Geometrically speaking, this may be interpreted more concretely as the fact that an anisotropic quadric cannot be ``rationally compressed'' to another quadric of ``sufficiently lower'' dimension. This striking behaviour was already observed in the aforementioned work of Hoffmann, who proved a less refined variant of Theorem \ref{thmKMT} in which the integer $\mathrm{dim}(p) - \witti{1}{p}$ is replaced by the largest power of 2 strictly less than $\mathrm{dim}(p)$ (at least when $\mathrm{char}(F)\neq 2$ - see \cite{Hoffmann1}). It should be noted that this observation has since gone on to have a great influence on developments in closely related topics within the theory of algebraic groups (see, e.g., \cite{Karpenko3}).

The main aim of the current article is to present a new upper bound for the integer $\witti{0}{q_{F(p)}}$ which is, on the one hand, complementary to that of Karpenko-Merkurjev-Totaro, and, on the other, is a direct generalization of Hoffmann's conjecture on the possible values of the first higher isotropy index. As such, it both strengthens and unifies an important part of the existing literature. As per Theorem \ref{thmKarpenko}, however, we must limit our considerations to quasilinear forms when working in the characteristic-$2$ setting; nevertheless, the results which we do obtain give a firm indication that the bound should hold without restriction. Our main theorem is thus the following; as the reader will readily observe, Theorem \ref{thmKarpenko} above is nothing else but the special case in which $q=p$:

\begin{theorem} \label{thmmain} Let $p$ and $q$ be anisotropic quadratic forms of dimension $\geq 2$ over a field $F$, and let $s = v_2\big(\mathrm{dim}(p) - \witti{1}{p}\big)$. Assume that either
\begin{enumerate} \item The characteristic of $F$ is not $2$, or
\item The characteristic of $F$ is $2$ and $q$ is quasilinear. \end{enumerate}
Then 
$$ \witti{0}{q_{F(p)}} \leq \mathrm{max}\left(\mathrm{dim}(q) - \mathrm{dim}(p),2^s\right). $$
\end{theorem}

Although we restrict to quasilinear forms in characteristic $2$, it is worth stressing that it is the proof of this case of Theorem \ref{thmmain} which is, to the author's mind, the most striking part of the paper. Furthermore, contrary to the more traditional order of events, the characteristic $\neq 2$ case of the theorem was in fact \emph{preceded} by the quasilinear case; indeed, the stated bound was only discovered by way of a conceptual trivialization established in the quasilinear setting. This is an effective advertisement for the study of quasilinear quadratic forms; intrinsic interest aside, the quasilinear case can sometimes serve as a guiding light for investigations into the general theory.

To give a another illustration of the unifying power of Theorem \ref{thmmain}, we now show how it may be used to give a short proof of a well-known and deep result in characteristic $\neq 2$ which was originally conjectured by Knebusch in \cite{Knebusch2} and proven by Fitzgerald in \cite{Fitzgerald};\footnote{An analogous result holds for quasilinear forms in characteristic 2 (see \cite[Prop. 1.7]{Laghribi2}), and this may also be deduced from Theorem \ref{thmmain} in a similar way.} by contrast, this result does \emph{not} follow directly from Theorems \ref{thmKMT} and \ref{thmKarpenko}.

\begin{example} (Fitzgerald, see \cite[Thm. 1.6]{Fitzgerald}) Let $p$ and $q$ be anisotropic quadratic forms of dimension $\geq 2$ over a field $F$ of characteristic $\neq 2$. If $q$ becomes hyperbolic over $F(p)$ $\big($i.e., $\witti{0}{q_{F(p)}} = \frac{1}{2}\mathrm{dim}(q)\big)$, and $\mathrm{dim}(p) > \frac{1}{2}\mathrm{dim}(q)$, then $q$ is similar to a Pfister form.

\begin{proof} Fitzgerald's original proof involved rather subtle manipulations of the Cassels-Pfister subform theorem (\cite[Thm. 22.5]{EKM}), among other standard results from the classical algebraic theory of quadratic forms. Given Theorem \ref{thmmain}, however, the statement follows in a straightforward manner. Indeed, since
$$ \witti{0}{q_{F(p)}} = \frac{1}{2}\mathrm{dim}(q) > \mathrm{dim}(q) - \mathrm{dim}(p) $$
by hypothesis, the theorem tells us in this case that
$$ \mathrm{dim}(q) = 2\witti{0}{q_{F(p)}} \leq 2^{s+1},$$ 
where $s = v_2\big(\mathrm{dim}(p) - \witti{1}{p}\big)$. Now, as $q$ becomes hyperbolic over $F(p)$, the Cassels-Pfister subform theorem implies that $\mathrm{dim}(p) \leq \mathrm{dim}(q)$. By the definition of $s$, we therefore have
\begin{equation} \label{fitz}  2^s < \mathrm{dim}(p),\mathrm{dim}(q) \leq 2^{s+1}. \end{equation}
Suppose now that $q$ is \emph{not} similar to a Pfister form. Then there exists an extension $L$ of $F$ such that (i) $q_L$ is isotropic and (ii) its anisotropic kernel $\anispart{(q_L)}$ \emph{is} similar to a Pfister form (for example, we can take $L$ to be the penultimate entry in the generic splitting tower of $q$ -- see \cite[\S 25]{EKM}). But $\anispart{(q_L)}$ becomes hyperbolic over $L(p_L)$, and so another application of the Cassels-Pfister subform theorem, together with \eqref{fitz}, shows that
$$ 2^s < \mathrm{dim}\big(\anispart{(q_L)}\big) < 2^{s+1}.$$
Since $\anispart{(q_L)}$, being similar to a Pfister form, has dimension equal to a power of $2$, this is not possible. We thus conclude that $q$ is similar to a Pfister form, as desired. \end{proof} \end{example}

The characteristic $\neq 2$ case of Theorem \ref{thmmain} is proved in \S \ref{secchar0} below (in fact, we get a slightly more refined assertion -- see Theorem \ref{thmmaininchar0}). The argument makes use of the theory of Chow correspondences in a similar way to the proofs of the aforementioned results of Karpenko-Merkurjev-Totaro and Karpenko, though we present it using the more conceptual language of Chow \emph{motives}. Here we rely crucially on the most recent advance in the study of motivic decompositions of quadrics due to Vishik (\cite{Vishik5}). Since Vishik's work involves heavy use of the aforementioned Steenrod operations on Chow groups modulo 2, this approach to Theorem \ref{thmmain} is certainly limited to characteristic $\neq 2$ at present. 

In order to treat the case where $\mathrm{char}(F) = 2$ and $q$ is quasilinear, we are therefore forced to proceed in an entirely different manner. Here we follow the pattern of ideas developed in \cite{Scully3}. Primary to Theorem \ref{thmmain}, our main result in this setting concerns a certain structural decomposition of the anisotropic kernel of the form $q_{F(p)}$ -- see Theorem \ref{thmsubformdec} below. This result (which may be viewed as a generalization of \cite[Thm. 5.1]{Scully3}) takes place on the level of quadratic forms themselves, and the quasilinear part of Theorem \ref{thmmain} emerges as nothing more than a dimension-theoretic consequence of its validity (again, we actually get a slightly better statement -- see Corollaries \ref{cormainboundquasilinear} and \ref{corquasilinearmainrefined}). Most interestingly, the quasilinear cases of Theorem \ref{thmKMT} and \ref{thmKarpenko} are also immediate dimension-theoretic consequences of Theorem \ref{thmsubformdec}, as are some additional new results which seem to have no known analogues in the characteristic $\neq 2$ theory, conjectural or otherwise (see \S \ref{secquasilinearapplications}). We are therefore left with a conceptually satisfying unification of several important statements in this setting, both old and new. Taking all of this into account, it would be interesting to know whether this algebraic approach to the isotropy problem (and Theorem \ref{thmsubformdec} in particular) admits some kind of generalization beyond the quasilinear case.

For the convenience of the reader, the proofs of the two cases of Theorem \ref{thmmain} are each preceded by a separate section of preliminaries particular to the relevant situation. Definitions, basic concepts and results common to both are given in the next section. 

Finally, we remark that, just as in \cite{Scully3}, the proofs of our main results on quasilinear quadratic forms readily generalize to give analogous results for quasilinear $p$-forms of any prime degree $p$ (i.e., Fermat-type forms of degree $p$ in characteristic $p$). For the sake of transparency, however, we refrain from working in this generality here. \\

\noindent {\bf Terminology.} In this text, the word \emph{scheme} means a scheme of finite-type over a field. The word \emph{variety} then means an integral scheme. 

\section{Diagonalizable quadratic forms} 

Let $F$ be a field and $V$ a finite-dimensional $F$-vector space. By a \emph{quadratic form on $V$}, we mean a map $q \colon V \rightarrow F$ such that (i) $q(\lambda v) = \lambda^2 q(v)$ for all $(\lambda,v) \in F \times V$, and (ii) the associated ``polar form'' $b_q \colon V \times V \rightarrow F$ given by $b_q(v,w) = q(v+w) - q(v) -q(w)$ is $F$-bilinear. If the $F$-linear map from $V$ to its dual space given by $v \mapsto b_q(v,-)$ is bijective, then we say that $q$ is \emph{non-degenerate}. The present article is concerned with quadratic forms which are \emph{diagonalizable}, in the sense that they arise as the diagonal part of a symmetric bilinear form over their field of definition. More precisely, $q$ is said to be \emph{diagonalizable} if there exists a symmetric bilinear form $b$ on $V_q$ such that $q(v) = b(v,v)$ for all $v \in V_q$. If $\mathrm{char}(F) \neq 2$, then every quadratic form over $F$ has this property; indeed, we can (and must) take $b = \frac{1}{2}b_q$ in this case. By contrast, if $\mathrm{char}(F) = 2$, then $q$ is diagonalizable if and only if it is \emph{quasilinear}, i.e., if and only if  $q(v+w) = q(v) + q(w)$ for all $v,w \in V_q$. As such our discussion will concern the entire theory of quadratic forms over fields of characteristic $\neq 2$, but only the quasilinear part of the characteristic-$2$ theory. \newline

\noindent {\bf Terminology.} In this section, by a \emph{quadratic form over $F$} (or simply a \emph{form over $F$}), we shall mean a \emph{diagonalizable} quadratic form on some finite-dimensional $F$-vector space.
\vspace{1 \baselineskip} 

We now recall some basic concepts and results concerning diagonalizable quadratic forms over arbitrary fields. The reader is referred to \cite{EKM} for a comprehensive discussion.

\subsection{Basic concepts} \label{secbasic} Let $q$ be a quadratic form over $F$. The $F$-vector space on which $q$ is defined will be denoted by $V_q$. The dimension of this vector space will be called the \emph{dimension} of $q$, and will be denoted by $\mathrm{dim}(q)$. The set $\lbrace q(v)\;|\;v \in V_q \rbrace$ consisting of all elements of $F$ represented by $q$ will be denoted by $D(q)$. Given a field extension $L$ of $F$, there is a unique quadratic form on the $L$-vector space $V_q \otimes_F L$ which extends $q$ and whose polar form extends $b_q$. We denote this form by $q_L$. If $R$ is a subring of $L$ containing $F$, then $D(q_R)$ will denote the subset $\lbrace q(w)\;|\;w \in V_q \otimes_F R \rbrace$ of $D(q_L) \cap R$. Given $a \in F^*$, we will write $aq$ for the quadratic form $v \mapsto aq(v)$ on the vector space $V_q$.

Let $p$ be another quadratic form over $F$. If there exists a bijective $F$-linear map $f \colon V_p \rightarrow V_q$ such that $q\big(f(v)\big) = p(v)$ for all $v \in V_p$, then we will say that $p$ and $q$ are \emph{isometric}, and write $p \simeq q$. If $p \simeq aq$ for some $a \in F^*$, then we will say that $p$ and $q$ are \emph{similar}. An element $a \in F^*$ for which $aq \simeq q$ is said to be a \emph{similarity factor} of $q$. The set $G(q)^* = \lbrace a \in F^*\;|\;aq \simeq q \rbrace$ of all similarity factors of $q$ is evidently a subgroup of $F^*$. We will write $G(q)$ for the union $G(q)^* \cup \lbrace 0 \rbrace$. Note that we have $F^2 \subseteq G(q)$.

The assignment $v + w \mapsto p(v) + q(w)$ \big($(v,w) \in V_p \times V_q$\big) defines a quadratic form $p \perp q$ on $V_p \oplus V_q$ called the \emph{orthogonal sum} of $p$ and $q$. Given a positive integer $n$, $n \cdot q$ will denote the orthogonal sum of $n$ copies of $q$ (note that $n\cdot q \neq nq$). If there exists a quadratic form $r$ over $F$ such that $q \simeq p \perp r$, then we will say that $p$ is a \emph{subform} of $q$. 

The \emph{tensor product} $p \otimes q$ is the unique quadratic form on $V_p \otimes V_q$ such that (i) the polar form of $p \otimes q$ is given by the Kronecker product $b_p \otimes b_q$, and (ii) for all $(v,w) \in V_p \times V_q$, $p \otimes q$ maps the element $v \otimes w$ to $p(v)q(w)$. If there exists a quadratic form $r$ over $F$ such that $q \simeq p \otimes r$, then we will say that $q$ is \emph{divisible} by $p$.

Given elements $a_1,\hdots,a_n \in F$, we write $\form{a_1,...,a_n}$ for the quadratic form $(\lambda_1,\hdots,\lambda_n) \mapsto \sum_{i=1}^n a_i\lambda_i^2$ on the $F$-vector space $F^{\oplus n}$. Every quadratic form over $F$ (in the sense we are considering) is clearly isometric to one of this type.

A vector $v \in V_q$ is said to be \emph{isotropic} (with respect to $q$) if $q(v) = 0$. An $F$-linear subspace of $V_q$ consisting entirely of isotropic vectors is said to be \emph{totally isotropic}. We will say that $q$ is \emph{isotropic} if $V_q$ contains a non-zero isotropic vector, and \emph{anisotropic} otherwise. Equivalently, $q$ is anisotropic if the projective quadric $Q = \lbrace q = 0\rbrace \subset \mathbb{P}(V_q)$ has no $F$-rational points. If $\mathrm{char}(F) \neq 2$, then this implies that $Q$ is smooth (\cite[Prop. 22.1]{EKM}). By contrast, if $\mathrm{char}(F) = 2$ and $q$ is non-zero, then $Q$ is \emph{totally singular} (in the sense that it has no smooth points at all), irrespective of whether $q$ is anisotropic or not. The \emph{isotropy index} of $q$, denoted $\mathfrak{i}_0(q)$, is defined to be the maximal dimension of a totally isotropic subspace of $V_q$. If $\mathrm{char}(F) \neq 2$, then $\mathfrak{i}_0(q)$ is more commonly known as the \emph{Witt index} of $q$. In any characteristic, the integer $\mathfrak{i}_0(q)$ is insensitive to making purely transcendental extensions of the base field (see \cite[Lem. 7.15]{EKM}).

\subsection{Decomposition of isotropic forms} \label{secisodecomp} To any quadratic form $q$ over $F$, one may associate an anisotropic quadratic form $\anispart{q}$ over $F$ known as the \emph{anisotropic part} of $q$. The precise nature of this form depends on the characteristic of $F$.

If $\mathrm{char}(F) \neq 2$, then the assignment $(x,y) \mapsto xy$ defines a quadratic form $\mathbb{H}$ on the $F$-vector space $F^{\oplus 2}$ known as the \emph{hyperbolic plane} over $F$. Up to isometry, $\mathbb{H}$ is the unique non-degenerate isotropic form of dimension 2 over $F$. If, in this case, $q$ is non-degenerate, then the \emph{Witt decomposition theorem} (see \cite[Thm. 8.5]{EKM}) characterizes $\anispart{q}$ as the unique anisotropic form over $F$ (up to isometry) such that $q \simeq \mathfrak{i}_0(q) \cdot \mathbb{H} \perp \anispart{q}$. Note, in particular, that we have $\mathrm{dim}(\anispart{q}) = \mathrm{dim}(q) -2\witti{0}{q}$ in this situation.

If $\mathrm{char}(F) = 2$ (so that $q$ is quasilinear), then the additivity of $q$ implies that the set $U$ of all isotropic vectors in $V_q$ is, in fact, an $F$-linear subspace of $V_q$. In particular, the isotropy index of $q$, $\witti{0}{q}$, is nothing else but the dimension of $U$. In this case, the form $\anispart{q}$ may be defined as the restriction of $q$ to the quotient space $V_q/U$. Up to isometry, $\anispart{q}$ is then the unique anisotropic form over $F$ such that $q \simeq \witti{0}{q} \cdot \form{0} \perp \anispart{q}$. From another point of view, the additivity of $q$ also implies that the set $D(q)$ consisting of all elements of $F$ represented by $q$ is an $F^2$-linear subspace of $F$. The form $\anispart{q}$ may then also be characterised up to isometry as the unique anisotropic form over $F$ such that $D(\anispart{q}) = D(q)$ (see Proposition \ref{propsubforms} below). It is important to note that in this setting we have $\mathrm{dim}(\anispart{q}) = \mathrm{dim}(q) - \witti{0}{q}$, as opposed to the aforementioned $\mathrm{dim}(\anispart{q}) = \mathrm{dim}(q) - 2\witti{0}{q}$ which prevails for non-degenerate forms in characteristic $\neq 2$.

Nevertheless, in all characteristics, we see that the dimension of $\anispart{q}$ measures the extent to which $q$ is isotropic. If $\mathrm{dim}(\anispart{q}) \leq 1$, then we say that $q$ is \emph{split}. Note that if $F$ is algebraically closed, then every quadratic form over $F$ is split. Moreover, if $\mathrm{char}(F) = 2$, then one only requires that $F$ be perfect in order to make the same conclusion. For example, there are no non-split forms over a finite field of characteristic 2.

\subsection{Function fields of quadrics} \label{secfunctionfields} Let $q$ be a quadratic form over $F$. If $q$ is not split, then it is irreducible as an element of the symmetric algebra $S(V_q^*)$, and so the projective quadric $Q$ is an integral $F$-scheme, as is its affine cone $\lbrace q = 0 \rbrace \subset \mathbb{A}(V_q)$ (see \cite[Ch. IV]{EKM}) In this case, we will write $F(q)$ for the function field of the former and $F[q]$ for that of the latter. The field $F[q]$ is of course a degree-1 purely transcendental extension of $F(q)$. If $L$ is a field extension of $F$, then we will simply write $L(q)$ instead of $L(q_L)$ whenever it is defined. Note that if $q \simeq \form{a_0,a_1,\hdots,a_n}$ for some $a_i \in F$ with $a_0,a_1 \neq 0$, then we have $F$-isomorphisms
$$F(q) \simeq F(S)\bigg(\sqrt{a_0^{-1}(a_1 + a_2S_2^2 + \cdots + a_nS_n^2)}\bigg)$$
and
$$ F[q] \simeq \mathrm{Frac}\big(F[T]/(a_0T_0^2 + \cdots +a_nT_n^2)\big) \simeq F(U)\left(\sqrt{a_0^{-1}(a_1U_1^2 + \cdots + a_nU_n^2)}\right),$$
where $S = (S_2,\hdots,S_n)$, $T = (T_0,\hdots,T_n)$ and $U = (U_1,\hdots,U_n)$ are systems of algebraically independent variables over $F$. Evidently, the form $q_{F(q)}$ is isotropic. 

\subsection{The Knebusch splitting pattern} \label{secsplittingtower} Let $q$ be a non-zero quadratic form over $F$. In \cite{Knebusch1}, Knebusch introduced the following construction (at least in the case where $\mathrm{char}(F) \neq 2$): Set $F_0 = F$, $q_0 = \anispart{q}$, and recursively define
\begin{itemize} \item $F_r = F_{r-1}(q_{r-1})$ (provided $q_{r-1}$ is not split), and
\item $q_r = \anispart{(q_{F_r})}$ (provided $F_r$ is defined). \end{itemize}
Note that if $q_r$ is defined, then $\mathrm{dim}(q_r) < \mathrm{dim}(q_{r-1})$, since an anisotropic quadratic form becomes isotropic over the function field of the quadric which it defines. As such, Knebusch's  process is finite, terminating at the first non-negative integer $h(q)$ for which $q_{h(q)}$ is split. The integer $h(q)$ is called the \emph{height} of $q$, and the tower of fields $F= F_0 \subset F_1 \subset \cdots \subset F_{h(q)}$ is called the \emph{Knebusch splitting tower} of $q$. For each $0 \leq r \leq h(q)$, we set $\wittj{r}{q} = \witti{0}{q_{F_r}}$. If $q$ is not split and $r \geq 1$, then the integer $\wittj{r}{q} - \wittj{r-1}{q}$ will be called the \emph{r-th higher isotropic index} of $q$, and will be denoted by $\witti{r}{q}$. In this case, the form $q_r$ will be called the \emph{r-th higher anisotropic kernel} of $q$. When $\mathrm{char}(F) \neq 2$, the tower $F = F_0 \subset F_1 \subset \cdots \subset F_{h(q)}$ is usually referred to as the \emph{generic splitting tower} of $q$, as it is generic (in a valuation-theoretic sense) among all towers of extensions of $F$ which ultimately split $q$. In particular, if $K$ is any field extension of $F$, then there exists in this case an integer $r \in [0,h(q)]$ such that $\witti{0}{q_K} = \witti{0}{q_{F_r}}$. This is easily deduced from the following more specific statement:

\begin{lemma}[{see \cite[\S 5]{Knebusch1}}] \label{lemgenerictower} Let $q$ be a quadratic form over a field $F$ of characteristic $\neq 2$ with Knebusch splitting tower $F = F_0 \subset F_1 \subset \cdots \subset F_{h(q)}$, and let $r$ be an integer in $[0,h(q)-1]$. If $K$ is a field extension of $F$ such that $\witti{0}{q_K} > \wittj{r}{q}$, then the free compositum $K \cdot F_{r+1}$ is a purely transcendental extension of $K$. \end{lemma}

Lemma \ref{lemgenerictower} follows from the fact that any smooth projective quadric which admits a rational point is a rational variety (\cite[Prop. 22.9]{EKM}). This is not the case for totally singular quadrics (see \cite[Rem. 7.4 (iii)]{Hoffmann2}), whence the characteristic restriction.

\subsection{Pfister and quasi-Pfister forms} \label{secPfister} Let $n$ be a positive integer. Given $a_1,\hdots,a_n \in F^*$, we write $\pfister{a_1,\hdots,a_n}$ for the $2^n$-dimensional quadratic form over $F$ defined as the $n$-fold tensor product $\form{1,-a_1} \otimes \cdots \otimes \form{1,-a_n}$. If $\mathrm{char}(F) \neq 2$, then forms of this type are known as (\emph{n-fold}) \emph{Pfister forms}. If $\mathrm{char}(F) =2$, they are more commonly referred to as (\emph{n-fold}) \emph{quasi-Pfister forms} (in order to distinguish them from the symmetric bilinear forms which also bear Pfister's name). In either case, any such form $\pi$ is \emph{round} (or \emph{multiplicative}), meaning that $G(\pi) = D(\pi)$. If $\mathrm{char}(F) = 2$, then the only round anisotropic forms over $F$ are those which are quasi-Pfister (\cite[Prop. 7.14]{Hoffmann2}). By contrast, if $\mathrm{char}(F) \neq 2$, there can be non-Pfister anisotropic round forms; in this case, the anisotropic Pfister forms over $F$ are precisely those anisotropic forms $\pi$ which are \emph{strongly multiplicative} in the sense that $G(\pi_K) = D(\pi_K)$ for \emph{every} field extension $K$ of $F$ (see \cite[Theorem 23.2]{EKM}).

\section{Some preliminaries in characteristic $\neq 2$} We begin our investigations with the case of fields of characteristic different from 2. Here we will make use of a certain interplay which exists in this setting between the Knebusch splitting pattern of an anisotropic quadratic form and some discrete motivic invariants of its associated quadric. We therefore begin by recalling some basic concepts and results concerning this interaction. For the remainder of this section, $F$ will denote an arbitrary field of characteristic $\neq 2$.

\subsection{Motivic decomposition type and upper motives} \label{secmotivic} For readable introductions to the theory of Chow motives, with particular emphasis on motives of quadrics, the reader is referred to \cite{Vishik4} or \cite{EKM}. If $k$ is a field, then we write $\textsf{\textsl{Chow}}(k)$ for the additive category of Grothendieck-Chow motives over $k$ with integral coefficients (as defined, for example, in \cite[\S 1]{Vishik4}). If $X$ is a smooth projective variety over $k$, then we will write $M(X)$ to denote its motive as an element of $\textsf{\textsl{Chow}}(k)$. In the special case where $X = \mathrm{Spec}(k)$, we simply write $\mathbb{Z}$ instead of $M(X)$, thus suppressing its dependency on the base field $k$. Given an integer $i$ and an object $M$ of $\textsf{\textsl{Chow}}(k)$, we will write $M\lbrace i \rbrace$ for the $i$-th Tate twist of $M$. In particular, $\mathbb{Z} \lbrace i \rbrace$ will denote the Tate motive with shift $i$ in $\textsf{\textsl{Chow}}(k)$. If $K$ is a field extension of $k$, we write $M_K$ to denote the image of an object $M$ in $\textsf{\textsl{Chow}}(k)$ under the natural scalar extension functor $\textsf{\textsl{Chow}}(k) \rightarrow \textsf{\textsl{Chow}}(K)$. 

Now, let $X$ be a smooth projective quadric of dimension $n \geq 1$ over our fixed field $F$ of characteristic $\neq 2$. By a result of Vishik (see \cite{Vishik1} or \cite[\S 3]{Vishik4}), any direct summand of $M(X)$ decomposes into a finite direct sum of indecomposable objects in $\textsf{\textsl{Chow}}(F)$. Furthermore, this decomposition is unique up to reordering of the summands. In the case where $X$ is split (i.e., $X$ is the vanishing locus of a split quadratic form over $F$), the full decomposition of $M(X)$ was obtained by Rost (\cite{Rost}), who showed that $M(X)$ decomposes into a certain direct sum of Tate motives. Since every smooth projective quadric over an algebraically closed field is split, it follows that $M(X_{\overline{F}})$ decomposes into a direct sum of Tate motives in $\textsf{\textsl{Chow}}(\overline{F})$, where $\overline{F}$ denotes a fixed algebraic closure of $F$. The precise statement is the following:
\begin{equation} \label{motdec} M(X_{\overline{F}}) \simeq \begin{cases} \bigoplus_{i=0}^n \mathbb{Z} \lbrace i \rbrace & \text{ if $n$ is odd} \\
                       \left(\bigoplus_{i=0}^n \mathbb{Z} \lbrace i \rbrace\right) \oplus \mathbb{Z}\lbrace \frac{n}{2} \rbrace & \text{ if $n$ is even} \end{cases} \end{equation}
(see \cite[Prop. 2.2]{Vishik4}). If $N$ is an indecomposable direct summand of $M(X)$, then it follows that $N_{\overline{F}}$ uniquely decomposes into a direct sum of a subset of the Tate motives appearing in the decomposition of $M(X_{\overline{F}})$. Moreover, the subsets which arise from distinct indecomposable summands of $M(Q)$ are necessarily disjoint (see \cite[\S 4]{Vishik4}). The complete motivic decomposition of $X$ over $F$ therefore determines a partition of the Tate motives which appear in \eqref{motdec}. This partition is an important discrete invariant of $X$ known as its \emph{motivic decomposition type}. The motivic decomposition type of $X$ interacts non-trivially with other known discrete invariants of $X$. In particular, it interacts with the Knebusch splitting pattern of its underlying form, and this interaction has been exploited to obtain deep results concerning the latter invariant in recent years. In order to prove the characteristic $\neq 2$ part of Theorem \ref{thmmain}, we will need the most recent advance in the study of the motivic decomposition type of anisotropic quadrics, which is a far-reaching result due to Vishik (\cite{Vishik5}). First, following \cite{Karpenko2}, we define the \emph{upper motive} of $X$ to be the unique indecomposable direct summand $U(X)$ of $M(X)$ such that the trivial Tate motive $\mathbb{Z}$ is isomorphic to a direct summand of $M(X)$. Vishik's result then gives the following information regarding $U(X)$ in the anisotropic case (we recall that an anisotropic quadric in characteristic $\neq 2$ is necessarily smooth -- see \S \ref{secbasic} above):

\begin{theorem}[{Vishik, see \cite[Thm. 2.1]{Vishik5}}] \label{thmexcellentconnections} Let $\phi$ be an anisotropic quadratic form of dimension $\geq 2$ over $F$ with associated $($smooth$)$ projective quadric $X$. Write
$$ \mathrm{dim}(\phi) - \witti{1}{\phi} = 2^{r_1} - 2^{r_2} + \cdots + (-1)^{t-1}2^{r_t} $$
for uniquely determined integers $r_1>r_2> \cdots >r_{t-1}>r_t + 1\geq 1$, and, for each $1 \leq l \leq t$, set
$$ D_l = \sum_{i=1}^{l-1} (-1)^{i-1}2^{r_i - 1} + \epsilon(l)\sum_{j=l}^t(-1)^{j-1}2^{r_j},$$
where $\epsilon(l) = 1$ if $l$ is even and $\epsilon(l) = 0$ if $l$ is odd. Then, for any $1 \leq l \leq t$, the Tate motive $\mathbb{Z}\lbrace D_l \rbrace$ is isomorphic to a direct summand of $U(X)_{\overline{F}}$.\end{theorem}

We will also make use the following more elementary observation (also due to Vishik) which relates two anisotropic quadrics on the motivic level in the situation where one of the quadrics is stably birational to the variety of totally isotropic subspaces of prescribed dimension in the other (see also \cite[Thm. 4.17]{Vishik4} for a strengthening of this result):

\begin{theorem}[{Vishik, cf. \cite[Prop. 1]{Vishik2}, \cite[Thm. 4.15]{Vishik4}}] \label{thmvishiksummands} Let $p$ and $q$ be anisotropic quadratic forms of dimension $\geq 2$ over $F$ with associated $($smooth$)$ projective quadrics $P$ and $Q$, respectively. Suppose that, for every field extension $K$ of $F$, we have
$$ \witti{0}{p_K} > 0 \;\;\; \Leftrightarrow \;\;\; \witti{0}{q_K} > l. $$
Then $U(P)\lbrace l \rbrace$ is isomorphic to a direct summand of $M(Q)$. \end{theorem}

\subsection{A criterion for stable birational equivalence of smooth projective quadrics}

Let $p$ and $q$ be anisotropic quadratic forms of dimension $\geq 2$ over $F$. If $q_{F(p)}$ is isotropic, then $\witti{0}{q_{F(p)}} \geq \witti{1}{q}$ (see \S \ref{secsplittingtower}), and so $\mathrm{dim}(q) - \witti{1}{q} \geq \mathrm{dim}(p) - \witti{1}{p}$ by an application of Theorem \ref{thmKMT}. In order to make use of Theorem \ref{thmvishiksummands} above in the situation of Theorem \ref{thmmain}, we will need the following complementary statement, also due to Karpenko and Merkurjev:

\begin{theorem}[{Karpenko-Merkurjev, see \cite[Thm. 4.1]{KarpenkoMerkurjev}}] \label{thmKMcriterion} Let $p$ and $q$ be anisotropic quadratic forms of dimension $\geq 2$ over $F$. If $q_{F(p)}$ is isotropic, and $\mathrm{dim}(q) - \witti{1}{q} = \mathrm{dim}(p) - \witti{1}{p}$, then $p_{F(q)}$ is isotropic as well. \end{theorem}

\begin{remark} In \cite{Totaro1}, Totaro showed that Theorem \ref{thmKMcriterion} is also valid in characteristic 2 (irrespective of any smoothness assumptions). We shall not need this result here. \end{remark}

\section{Main results: The characteristic $\neq 2$ case} \label{secchar0} In this section, we show that Theorem \ref{thmmain} is valid in characteristic $\neq 2$. Following the previous section, $F$ will denote an arbitrary field of characteristic $\neq 2$ throughout. Note that if $p$ is an anisotropic quadratic form of dimension $\geq 2$ over $F$, then Theorem \ref{thmKarpenko} asserts that the integer $\witti{1}{p} - 2^{v_2(\mathrm{dim}(p) - \witti{1}{p})}$ is non-positive. The characteristic $\neq 2$ case of Theorem \ref{thmmain} therefore follows from the following more refined statement:

\begin{theorem} \label{thmmaininchar0} Let $p$ and $q$ be anisotropic quadratic forms of dimension $\geq 2$ over $F$, and let $s = v_2\big(\mathrm{dim}(p) - \witti{1}{p}\big)$. Then
$$ \witti{0}{q_{F(p)}} \leq \mathrm{max}\left(\mathrm{dim}(q) - \mathrm{dim}(p) + \witti{1}{p} - 2^{s}, 2^{s}\right). $$
\begin{proof} In order to simplify the notation, set $\mathfrak{i} := \witti{0}{q_{F(p)}}$. Our assertion then reads
\begin{equation} \label{mainineq}  \mathfrak{i} \leq \mathrm{max}\left(\mathrm{dim}(q) - \mathrm{dim}(p) + \witti{1}{p} - 2^{s}, 2^{s}\right). \end{equation}
The idea of the proof is now the following: Let $P$ and $Q$ be the (necessarily smooth) projective $F$-quadrics defined by the vanishing of $p$ and $q$, respectively. Assuming that \eqref{mainineq} fails to holds, we show (modulo an inductive assumption) that a certain shift of the upper motive of $P$ is a direct summand of the motive of $Q$ (see \S \ref{secmotivic} above). We then show that this is not possible by using Theorem \ref{thmexcellentconnections} to analyse the motivic decomposition types of $P$ and $Q$. The key observation here is the following:

\begin{lemma} \label{lemchar0calc} Suppose, in the above situation, that \eqref{mainineq} does not hold. Then, for any field extension $K$ of $F$, we have
$$ \witti{0}{p_K} > 0 \;\;\; \Leftrightarrow \;\;\; \witti{0}{q_K} \geq \mathfrak{i}. $$
\begin{proof} Let $F = F_0 \subset F_1 \subset \cdots \subset F_{h(q)}$ be the Knebusch splitting tower of $q$ and let $q_1,\hdots,q_{h(q)}$ be its higher anisotropic kernel forms (see \S \ref{secsplittingtower}). Recall that for each integer $k \in [0,h(q)]$, $\wittj{k}{q}$ denotes the isotropy index $\witti{0}{q_{F_k}}$ of $q$ over $F_k$. Since \eqref{mainineq} fails to hold by hypothesis, we have $\mathfrak{i} > 1$. In particular, $q_{F(p)}$ is isotropic. Thus, by the genericity of Knebusch's construction, there exists  an integer $r \in [0,h(q)-1]$ such that $\wittj{r+1}{q} = \mathfrak{i}$ (again, see \S \ref{secsplittingtower}). We claim that $p_{F_{r+1}}$ is isotropic. In order to prove this claim, we may clearly assume that $p_{F_r}$ is anisotropic (in fact, one can show that this is indeed the case, but we don't need to know that here). Working under this assumption, let us now set
$$ \mu := \mathrm{dim}(q_r) - \witti{1}{q_r} - \big(\mathrm{dim}(p_{F_r}) - \witti{1}{p_{F_r}}\big). $$
Since $\mathfrak{i} = \wittj{r+1}{q}> \wittj{r}{q}$, $q_r$ becomes isotropic over $F_r(p)$. As $F_{r+1} = F_r(q_r)$ by definition, Karpenko and Merkurjev's Theorem \ref{thmKMcriterion} shows that, to prove our claim, it suffices to check that $\mu = 0$. By the remarks preceding Theorem \ref{thmKMcriterion}, we certainly have $\mu \geq 0$. To see that equality holds here, let us first note that $F_r(p)$ is a purely transcendental extension of $F(p)$ by Lemma \ref{lemgenerictower}. Since isotropy indices are insensitive to purely transcendental extensions (\S \ref{secbasic}), it follows that $\witti{1}{p_{F_r}} = \witti{1}{p}$. In particular, we have
\begin{equation} \label{mudesc} \mu = \mathrm{dim}(q_r) - \witti{1}{q_r} - \big(\mathrm{dim}(p) - \witti{1}{p}\big). \end{equation}
Now, since $\mathfrak{i} > \mathrm{dim}(q) - \mathrm{dim}(p) + \witti{1}{p} - 2^s$ by hypothesis, we see that
\begin{eqnarray*} \big(\mathrm{dim}(p) - \witti{1}{p}\big) + \mu & = &  \mathrm{dim}(q_r) - \witti{1}{q_r} \\
&=& \mathrm{dim}(q) - 2\mathfrak{j}_r(q) - \witti{r+1}{q} \\
&=& \mathrm{dim}(q) - \mathfrak{j}_{r+1}(q) - \mathfrak{j}_r(q) \\
&=& \mathrm{dim}(q) - \mathfrak{i} - \mathfrak{j}_r(q) \\
&<& \big(\mathrm{dim}(p) - \witti{1}{p}\big) + 2^s - \mathfrak{j}_r(q), \end{eqnarray*}
and so $\mu < 2^s - \mathfrak{j}_r(q)$. On the other hand, since we are also assuming that $\mathfrak{i} > 2^s$, we have
$$ \witti{1}{q_r} = \mathfrak{j}_{r+1}(q) - \mathfrak{j}_r(q) = \mathfrak{i} - \mathfrak{j}_r(q) > 2^s - \mathfrak{j}_r(q). $$
By Karpenko's theorem (see Theorem \ref{thmKarpenko}), it follows that $\mathrm{dim}(q_r) - \witti{1}{q_r}$ is divisible by $2^t$ for some integer $t$ satisfying $2^t \geq \witti{1}{q_r} > 2^s - \mathfrak{j}_r(q)$. But, by \eqref{mudesc}, we have
$$ \mathrm{dim}(q_r) - \witti{1}{q_r} = \big(\mathrm{dim}(p) - \witti{1}{p}\big) + \mu. $$
As $\mathrm{dim}(p) - \witti{1}{p}$ is (by definition) divisible by $2^s$, it follows that $\mu$ is divisible by $\mathrm{min}(2^s, 2^t)$. Since $0 \leq \mu < 2^s - \mathfrak{j}_r(q) < \mathrm{min}(2^t,2^s)$, this shows that $\mu = 0$, and thus proves our initial claim that $p_{F_{r+1}}$ is isotropic. Finally, to complete the proof, let $K$ be any field extension of $F$. If $\witti{0}{p_K}>0$, then $K(p)$ is, by Lemma \ref{lemgenerictower}, a purely transcendental extension of $K$, and so $\witti{0}{q_K} = \witti{0}{q_{K(p)}} \geq \witti{0}{q_{F(p)}} \geq \mathfrak{i}$. Conversely, if $\witti{0}{q_K} \geq \mathfrak{i}$, then Lemma \ref{lemgenerictower} again shows that the compositum $L :=K\cdot F_{r+1}$ is a purely transcendental extension of $K$, and so $\witti{0}{p_K}  = \witti{0}{p_L} \geq \witti{0}{p_{F_{r+1}}} > 0$. The conditions $\witti{0}{p_K} > 0$ and $\witti{0}{q_K} \geq \mathfrak{i}$ are therefore equivalent, which is what we wanted to prove. \end{proof}
\end{lemma}
We now return to the proof of Theorem \ref{thmmaininchar0}. We will argue by induction on $\mathrm{dim}(q)$, the case where $\mathrm{dim}(q) = 2$ being trivial. Assume now that $\mathrm{dim}(q) > 2$ and that the theorem is valid for all forms of dimension $< \mathrm{dim}(q)$ over $F$. Under this inductive hypothesis, we will show that \eqref{mainineq} is valid for $q$. For this, we will use the following trivial observation:

\begin{lemma} \label{lemsubformreduction} Let $q$ be a quadratic form over $F$ and let $K$ be a field extension of $F$. Then, for any integer $i \in [0,\witti{0}{q_K}]$, there exists a subform $q' \subset q$ such 
\begin{enumerate} \item $\mathrm{dim}(q') \leq \mathrm{dim}(q) - i$.
\item $\witti{0}{q'_K} = \witti{0}{q_K} - i$. \end{enumerate}
\begin{proof} It is enough to treat the case where $i = 1$. We argue in this case by induction on $\mathrm{dim}(q)$. The case where $\mathrm{dim}(q) =0$ is trivial. Assume now that $\mathrm{dim}(q) > 0$, and let $U$ be a totally isotropic $K$-linear subspace of $V_q \otimes_F K$ of dimension $\witti{0}{q_K}$. If $W$ is a hyperplane in $V_q$, then the intersection of $U$ and $W \otimes_F K$ in $V_q \otimes_F K$ has $K$-dimension at least $\witti{0}{q_K} - 1$. In particular, if $r$ is a codimension-1 subform of $q$, then $\witti{0}{r_K} \geq \witti{0}{q_K} -1$. By the induction hypothesis, $r$ admits a subform $q'$ of dimension $\leq \mathrm{dim}(q) - i$ such that $\witti{0}{q'_K} = \witti{0}{q_K} - i$. Since $r$ is a subform of $q$, so is $q'$, and the result therefore follows.   \end{proof} \end{lemma}

Now, suppose, for the sake of contradiction, that inequality \eqref{mainineq} is not valid for $q$. Then $\mathfrak{i}>2^s$, and so, by Lemma \ref{lemsubformreduction}, there exists a subform $q'$ of codimension at least $\mathfrak{i} - (2^s +1)$ in $q$ such that $\witti{0}{q'_{F(p)}} = 2^s + 1$. But then the statement of our theorem fails to hold for the pair $(q',p)$. By our inductive assumption, we therefore conclude that $q' = q$, and so $\mathfrak{i} = 2^s + 1$. Now, since \eqref{mainineq} fails to hold, Lemma \ref{lemchar0calc} shows that the condition of Theorem \ref{thmvishiksummands} is satisfied with $l = \mathfrak{i}-1 = 2^s$, and hence $U(P)\lbrace 2^s \rbrace$ is isomorphic to a direct summand of $M(Q)$. Note, however, that $U(P)\lbrace 2^s \rbrace$ is \emph{not} isomorphic to the upper motive $U(Q)$ of $Q$. Indeed, since $2^s \geq 1$, the trivial Tate motive $\mathbb{Z}$ is not isomorphic to a direct summand of $U(P)\lbrace 2^s \rbrace_{\overline{F}}$. Thus, $U(Q)$ and $U(P)\lbrace 2^s \rbrace$ are (isomorphic to) \emph{distinct} indecomposable direct summands of $M(Q)$. To complete the proof, we will now obtain a contradiction to our initial supposition by showing that $U(P)\lbrace 2^s \rbrace_{\overline{F}}$ and $U(Q)_{\overline{F}}$ have a common Tate motive in their respective decompositions (this cannot happen for non-isomorphic direct summands of $M(Q)$ -- see \S \ref{secmotivic} above). For this, we use Vishik's Theorem \ref{thmexcellentconnections}. More specifically, we will use Vishik's result to show that $\mathbb{Z}\lbrace m \rbrace$ is isomorphic to a direct summand of both $U(P)\lbrace 2^s \rbrace_{\overline{F}}$ and $U(Q)_{\overline{F}}$, where $m := (\mathrm{dim}(p) - \witti{1}{p} + 2^s)/2$. In the former case, this follows immediately from Theorem \ref{thmexcellentconnections}. Indeed, the reader will quickly check that, for $X = P$, the $l=t$ case of the latter result states precisely that $\mathbb{Z}\lbrace m - 2^s \rbrace$ is isomorphic to a direct summand of $U(P)_{\overline{F}}$. To see that $\mathbb{Z} \lbrace m \rbrace$ is isomorphic to a direct summand of $U(Q)_{\overline{F}}$, however, we first need to analyse the alternating $2$-expansion of the integer $\mathrm{dim}(q) - \witti{1}{q}$. More specifically, we need the following observation: 

\begin{lemma} \label{lemaltexp} Suppose we are in the above situation, and write
$$ \mathrm{dim}(q) - \witti{1}{q} = 2^{r_1} - 2^{r_2} + \cdots + (-1)^{t-1}2^{r_t} $$
for uniquely determined integers $r_1>r_2> \cdots>r_{t-1}>r_t + 1 \geq 1$. Then there exists an even integer $k \in [1,t-1]$ such that
$$ 2m = 2^{r_1} - 2^{r_2} + \cdots + 2^{r_{k-1}} - 2^{r_k}. $$ \end{lemma}

Given this assertion, we obtain that $\mathbb{Z}\lbrace m \rbrace$ is isomorphic to a direct summand of $U(Q)_{\overline{F}}$ by applying Theorem \ref{thmexcellentconnections} to the case where $X = Q$ and $l = k+1$. Before proving Lemma \ref{lemaltexp}, we make an intermediate calculation:

\begin{sublemma} \label{sublemalpha} In the above situation, the integer
$$\alpha := \big(\mathrm{dim}(q) - \mathrm{dim}(p) + \witti{1}{p} - 2^s\big) - \witti{1}{q}$$
is positive.
\begin{proof} Let $F = F_0 \subset F_1 \subset \cdots \subset F_{h(q)}$ be the Knebusch splitting tower of $q$, and let $r \in [0,h(q) - 1]$ be such that $\wittj{r+1}{q} = \mathfrak{i}$ (see \S \ref{secsplittingtower}). Since we are assuming that \eqref{mainineq} does not hold, the proof of Lemma \ref{lemchar0calc} shows that
\begin{eqnarray}  \notag\mathrm{dim}(p) - \witti{1}{p} &=& \mathrm{dim}(q_r) - \witti{1}{q_r} \\
 \label{eqsublem} &=& \mathrm{dim}(q) - \mathfrak{i} - \wittj{r}{q}\\ 
 \notag & =& \mathrm{dim}(q) - (2^s + 1) - \wittj{r}{q}, \end{eqnarray}
and so $\alpha = \wittj{r}{q} - \witti{1}{q} + 1$. Since $\wittj{r}{q} \geq \witti{1}{q}$ if and only if $r \geq 1$, the claim therefore amounts to the assertion that $r \neq 0$. Suppose, for the sake of contradiction, that this is not the case. Then we have
$$ \witti{1}{q} = \wittj{1}{q} = \mathfrak{i} = 2^s + 1. $$
On the other hand, since $\wittj{0}{q} = 0$, \eqref{eqsublem} then becomes
$$ \mathrm{dim}(q) - \witti{1}{q} = \mathrm{dim}(p) - \witti{1}{p}, $$
and so $\witti{1}{q} \leq 2^s $ by Karpenko's theorem (see Theorem \ref{thmKarpenko}) and the definition of $s$. We thus conclude that our supposition was incorrect, and so our claim is valid.
\end{proof} \end{sublemma}

\begin{proof}[Proof of Lemma \ref{lemaltexp}] We begin by observing that we can write
\begin{equation} \label{alphaeq} \mathrm{dim}(q) - \witti{1}{q} = \big(\mathrm{dim}(p) - \witti{1}{p} + 2^s\big) + \alpha, \end{equation}
where, as above,
$$ \alpha = \big(\mathrm{dim}(q) - \mathrm{dim}(p) + \witti{1}{p} - 2^s\big) - \witti{1}{q}.$$
By Sublemma \ref{sublemalpha}, $\alpha$ is positive. On the other hand, since \eqref{mainineq} fails to hold, the integer $\big(\mathrm{dim}(q) - \mathrm{dim}(p) + \witti{1}{p} - 2^s\big)$ is strictly less than $\mathfrak{i} = 2^s + 1$. Since $\witti{1}{q} \geq 1$, we conclude that $0 < \alpha < 2^s$. Now, by the very definition of $s$, we have 
$$ \mathrm{dim}(p) - \witti{1}{p} = 2^{b_1} - 2^{b_2} + \cdots + (-1)^{u-2}2^{b_{u-1}} + (-1)^{u-1}2^{s} $$
for unique integers $b_1 > b_2 > \cdots > b_{u-1} > s + 1 \geq 1$. Adding $2^s$, we see that
$$ \mathrm{dim}(p) - \witti{1}{p} + 2^s = \begin{cases} 2^{b_1} - 2^{b_2} + \cdots + 2^{b_{u-1}} & \text{ if $u$ is even}\\
2^{b_1} - 2^{b_2} + \cdots - 2^{b_{u-1}} + 2^{s+1} & \text{ if $u$ is odd.} \end{cases} $$
Irrespective of whether $u$ is odd or even, it follows that
$$ \mathrm{dim}(p) - \witti{1}{p} + 2^s = 2^{r_1} - 2^{r_2} + \cdots + 2^{r_{k-1}} - 2^{r_k} $$
for some even integer $k$ and integers $r_1>r_2>\cdots>r_{k-1}>r_k>s$. At the same time, since $0 < \alpha < 2^s$, we can write
$$ \alpha = 2^{r_{k+1}} - 2^{r_{k+2}} + \cdots + (-1)^{t-1}2^{r_t} $$
for uniquely determined integers $s \geq r_{k+1} > r_{k+2} > \cdots > r_{t-1} > r_t + 1 \geq 1$. By \eqref{alphaeq}, we then have
$$ \mathrm{dim}(q) - \witti{1}{q} = \big(2^{r_1} - 2^{r_2} + \cdots + 2^{r_{k-1}} - 2^{r_k}\big) + \big(2^{r_{k+1}} - 2^{r_{k+2}} + \cdots + (-1)^{t-1}2^{r_t}\big), $$
and this is precisely the description of $\mathrm{dim}(q) - \witti{1}{q}$ as an alternating sum of $2$-powers which appears in the statement of the lemma. Since $m = (\mathrm{dim}(p) - \witti{1}{p} + 2^s)/2$ by definition, we are done. \end{proof}

With the lemma proved, we have shown that the non-isomorphic indecomposable direct summands $U(Q)$ and $U(P)\lbrace 2^s \rbrace$ of $M(Q)$ admit a common Tate motive in their respective decompositions after scalar extension to $\overline{F}$, thus providing us with the needed contradiction. This completes the induction step and so the theorem is proved. \end{proof}
\end{theorem}

\begin{remark} Theorem \ref{thmmaininchar0} is indeed a non-trivial refinement of Theorem \ref{thmmain} in the characteristic $\neq 2$ setting. For example, if $p$ is odd-dimensional, then the integer $\witti{1}{p} - 2^s$ is negative by Theorem \ref{thmKarpenko}, and so we get a slightly better result in this case. \end{remark}

\section{Some preliminaries in characteristic 2} We now turn to the case of quasilinear quadratic forms in characteristic 2. We begin by collecting some preliminary results to be used later on. For detailed expositions of the basic theory of quasilinear quadratic forms, the reader is referred to \cite{Hoffmann2} and \cite{Scully3}. For the remainder of this section, $F$ will now denote an arbitrary field of \emph{characteristic $2$}.

\subsection{The representation theorem} It is well known that an anisotropic quadratic form over an arbitrary field can be recovered (up to isometry) from the set of values that it represents over \emph{every} extension of that field (see \cite[Thm. 17.12]{EKM}) In the quasilinear setting, an anisotropic form can already be recovered from the set of values that it represents over the \emph{base field}. Indeed, we have the following simple observation:

\begin{proposition}[{cf. \cite[Prop. 8.1]{HoffmannLaghribi1}}] \label{propsubforms} Let $p$ and $q$ be quasilinear quadratic forms over $F$. Then $\anispart{p} \subset \anispart{q}$ if and only if $D(p) \subseteq D(q)$. In particular, $\anispart{p} \simeq \anispart{q}$ if and only if $D(p) = D(q)$. \end{proposition}

We will need the following easy consequence of Proposition \ref{propsubforms}:

\begin{corollary}[{cf. \cite[Prop. 8.1]{HoffmannLaghribi1}}] \label{corexcellence}Let $q$ be a quasilinear quadratic form over $F$ and let $L$ be a field extension of $F$. Then there exists a subform $r \subset \anispart{q}$ such that $r_L \simeq \anispart{(q_L)}$.
\begin{proof} $D(q_L)$ is obviously spanned as an $L^2$-vector space by elements of $D(q)$. If $a_1,\hdots,a_n \in D(q)$ form a basis of $D(q_L)$ over $L^2$, then Proposition \ref{propsubforms} shows that $r = \form{a_1,\hdots,a_n}$ is a subform of $\anispart{q}$ having the required property. \end{proof} \end{corollary}

\subsection{Inseparable quadratic extensions} We will make use of some basic results concerning the behaviour of quasilinear quadratic forms under scalar extension to an inseparable quadratic extension of the ground field. We first note the following statement, which is a direct consequence of the additivity property of quasilinear quadratic forms.

\begin{lemma}[{cf. \cite[Lem. 3.8]{Scully2}}] \label{lemvaluesquadratic} Let $q$ be a quasilinear quadratic form over $F$, and let $K = F(\sqrt{a})$ for some $a \in F \setminus F^2$. Then $D(q_K) = D(q) + aD(q)$ $($as subsets of $K)$. \end{lemma}

The following result elaborates upon Lemma \ref{corexcellence} in the case where $L$ is an inseparable quadratic extension of $F$, and is directly analogous to a standard result in the characteristic $\neq 2$ theory (see \cite[Cor. 22.12]{EKM}):

\begin{lemma}[{cf. \cite[Prop. 5.10]{Hoffmann2}}] \label{lemquadraticextensions} Let $q$ be an anisotropic quasilinear quadratic form over $F$ and let $K = F(\sqrt{a})$ for some $a \in F \setminus F^2$. If $\witti{0}{q_K} = n$, and $r \subset q$ is such that $r_{K} \simeq \anispart{(q_K)}$, then there exist elements $b_1,\hdots,b_n \in D(r)$ such that 
$$ q \simeq r \perp a\form{b_1,\hdots,b_n}. $$ \end{lemma}

\subsection{Quasi-Pfister forms} \label{secqpforms} Let $\pi$ be a quasi-Pfister form over $F$ (see \S \ref{secPfister} above). Then $\pi$ is round, i.e, $D(q) = G(q)$. In particular, the set $D(q)$ is not only an $F^2$-linear subspace of $F$, it is an $F^2$-linear \emph{subfield} of $F$ (recall that the similarity factors of any form over $F$ constitute a subgroup of $F^*$). Moreover, this property characterises anisotropic quasi-Pfister forms over $F$, since such forms are characterised by their roundness.

\subsection{The norm form} (see \cite[\S 8]{HoffmannLaghribi1}, \cite[\S 4]{Hoffmann2}) \label{secnormform} Let $q$ be a quasilinear quadratic form over $F$. The \emph{norm field} of $q$, denoted $N(q)$, is defined as the smallest subfield of $F$ containing all ratios of non-zero elements of $D(q)$. Evidently, we have $N(aq) = N(q)$ for all $a \in F^*$. Note also that $F^2 \subseteq N(q)$ by definition. Thus, by \S \ref{secqpforms}, there exists, up to isometry, a unique anisotropic quasi-Pfister form $\normform{q}$ over $F$ with the property that $D(\normform{q}) = N(q)$. This form is called the \emph{norm form} of $q$. Note that if $a$ is any non-zero element of $F$ represented by $q$, then $a\anispart{q} \subset \normform{q}$. Indeed, this follows from Proposition \ref{propsubforms} in light of the obvious inclusion $D(a\anispart{q}) = aD(q) \subseteq N(q) = D(\normform{q})$. The dimension of $\normform{q}$ is an important invariant known as the \emph{norm degree} of $q$. In what follows, it will be more convenient to work with its base-2 logarithm which we denote by $\mathrm{lndeg}(q)$. In other words, $\mathrm{lndeg}(q) = \mathrm{log}_2\big(\mathrm{dim}(\normform{q})\big) = \mathrm{log}_2[N(q):F^2]$. We will need the following lemma:

\begin{lemma}[{cf. \cite[Thm. 8.11 (i)]{HoffmannLaghribi1}}] \label{lemndegtower} Let $p$ be a quasilinear quadratic form over $F$. Then $\mathrm{lndeg}(p_i) = \mathrm{lndeg}(p) - i$ for all $1 \leq i \leq h(p)$. \end{lemma}

\subsection{Similarity factors} \label{secsimilarity} (see \cite[\S 6]{Hoffmann2}) Another direct consequence of Proposition \ref{propsubforms} is the following statement concerning similarity factors:

\begin{lemma}[{cf. \cite[Lem. 6.3]{Hoffmann2}}] \label{lemsimilarity} Let $q$ be a quasilinear quadratic form over $F$, and let $a \in F^*$. Then $a \in G(q)$ if and only if $aD(q) = D(q)$. \end{lemma}

Note that because $q$ is quasilinear, the condition of the preceding lemma is closed under addition. Since $G(q) \setminus \lbrace 0 \rbrace$ is a subgroup of $F^*$, it follows that $G(q)$ is a subfield of $F$.

\subsection{Divisibility by quasi-Pfister forms} \label{secdivbyqp} Our main results on the isotropy behaviour of quasilinear quadratic forms over function fields of quadrics will be obtained by studying the extent to which certain forms are divisible by quasi-Pfister forms. With this in mind, it is useful to introduce some related terminology (see \cite[\S 2.11]{Scully3}): 

Given a quasilinear quadratic form $q$ over $F$, we define its \emph{divisibility index}, denoted $\mathfrak{d}_0(q)$, to be the largest non-negative integer $s$ such that $\anispart{q}$ is divisible by an $s$-fold quasi-Pfister form. The \emph{higher divisibility indices} of $q$, $\mathfrak{d}_1(q), \mathfrak{d}_2(q),\hdots,\mathfrak{d}_{h(q)}(q)$, are defined as the divisibility indices of the higher anisotropic kernel forms $q_1,q_2,\hdots,q_{h(q)}$, respectively. It will be worth recording the following easy calculation from \cite{Scully3}:

\begin{lemma}[{cf. \cite[Lem. 2.35]{Scully3}}] \label{lemdivinvariance} Let $q$ be a quasilinear quadratic form over $F$ and let $L$ be a purely transcendental field extension of $F$. Then $\mathfrak{d}_0(q_L) = \mathfrak{d}_0(q)$. In particular, the higher divisibility indices of any quasilinear quadratic form are insensitive to purely transcendental field extensions. \end{lemma}

We will need some further observations regarding the above notion of divisibility. The first of these is analogous to a standard result (\cite[Cor. 2.1.11]{Kahn}) regarding divisibility of non-degenerate quadratic forms by Pfister forms over fields of characteristic $\neq 2$:

\begin{lemma}[{cf. \cite[Prop. 4.19]{Hoffmann2}}] \label{lemanisdivisibility} Let $q$ be a quasilinear quadratic form and $\pi$ an anisotropic quasi-Pfister form over $F$. If $q$ is divisible by $\pi$, then $\anispart{q}$ is also divisible by $\pi$. \end{lemma}

Next, it is well known that if $\mathrm{char}(F) \neq 2$, then an anisotropic quadratic form $q$ over $F$ is divisible by an anisotropic Pfister form $\pi$ if and only if $D(\pi_K) \subseteq G(q_K)$ for \emph{every} field extension $K/F$ (see \cite[Thm. 20.16 and Cor. 23.6]{EKM}). In the quasilinear setting, this condition can already be checked over the \emph{base field} (compare Proposition \ref{propsubforms}):

\begin{proposition}[{cf. \cite[Prop. 6.4]{Hoffmann2}, \cite[Cor. 2.20]{Scully3}}] \label{propdivisibility} Let $q$ be an anisotropic quasilinear quadratic form and $\pi$ an anisotropic quasi-Pfister form over $F$. Then $q$ is divisible by $\pi$ if and only if $D(\pi) \subseteq G(q)$. \end{proposition}

This has the following consequence:

\begin{corollary} \label{cordivbynormform} Let $q$ and $p$ be anisotropic quasilinear quadratic forms over $F$. Then $q$ is divisible by $\normform{p}$ if and only if $\anispart{(q \otimes p)}$ is similar to $q$.
\begin{proof} Let $a_0,a_1,\hdots,a_n \in F^*$ be such that $p \simeq \form{a_0,a_1,\hdots,a_n}$. By Proposition \ref{propdivisibility} and the definition of $\normform{p}$, $q$ is divisible by $\normform{p}$ if and only if $N(p) \subseteq G(q)$. But $N(p) = F^2\left(\frac{a_1}{a_0},\hdots,\frac{a_n}{a_0}\right)$, and, since $G(q)$ is a subfield of $F$ containing $F^2$ (see \S \ref{secsimilarity}), we see that $q$ is divisible by $\normform{p}$ if and only if $a_ia_0^{-1} \in G(q)$ for all $1 \leq i \leq n$. By Lemma \ref{lemsimilarity}, this is equivalent to the assertion that $a_ia_0^{-1}D(q) = D(q)$ for all $1 \leq i \leq n$. Since $D(q)$ is a finite-dimensional $F^2$-linear subspace of $F$, these equalities hold if and only if
$$ \sum_{i=0}^n a_ia_0^{-1}D(q) \subseteq D(q).$$
But $\sum_{i=0}^n a_ia_0^{-1}D(q)=D(a_0^{-1}q \otimes p)$, so by Proposition \ref{propsubforms}, the preceding inclusion amounts to the assertion that $\anispart{(q \otimes p)} \subset a_0q$. On the other hand, Proposition \ref{propsubforms} also shows that $a_0q$ is a subform of $\anispart{(q \otimes p)}$. Thus, we see that $q$ is divisible by $\normform{p}$ if and only if $a_0q \simeq \anispart{q \otimes p}$, which proves the desired assertion. \end{proof} \end{corollary}

\subsection{Isotropy of quasilinear quadratic forms under field extensions} \label{secquasilinearisotropy} We record here some basic observations regarding the isotropy behaviour of quasilinear quadratic forms under field extensions. The first concerns \emph{separable} extensions (here, when we say that $L$ is a separable extension of $F$, we mean that for any algebraic closure $\overline{F}$ of $F$, the ring $L \otimes_F \overline{F}$ has no nilpotent elements):

\begin{lemma}[{cf. \cite[Prop. 5.3]{Hoffmann2}}] \label{lemseparable} Let $q$ be an anisotropic quasilinear quadratic form over $F$. If $L$ is a separable field extension of $F$, then $q_L$ is anisotropic.
\end{lemma}

In particular, anisotropic forms remain anisotropic under purely transcendental field extensions. Secondly, we will need a slightly more subtle result regarding function fields of (affine) hypersurfaces. Let $T = (T_1,\hdots,T_m)$ be a system of algebraically independent variables over $F$, let $g \in F[T]$ be an \emph{irreducible} polynomial and let $F[g]$ denote the fraction field of the integral domain $F[T]/(g)$ (i.e., the function field of the integral affine hypersurface $\lbrace g = 0 \rbrace \subset \mathbb{A}_F^m$). Given an element $f \in F[T]$, let us write $\mathrm{mult}_g(f)$ for the largest non-negative integer $s$ such that $f = g^sh$ for some $h \in F[T]$ (with the added convention that $\mathrm{mult}_g(0) = +\infty$). Then we have the following:

\begin{lemma}[{cf. \cite[Prop. 2.33]{Scully3}}] \label{lemmultiplicity} In the above situation, let $q$ be a quasilinear quadratic form and let $f \in F[T]$. Suppose that $f \in D(q_{F(T)})$ and that $q_{F[g]}$ is anisotropic. Then $\mathrm{mult}_g(f) \equiv 0 \pmod{2}$. \end{lemma}

For $m=1$, the assertion of Lemma \ref{lemmultiplicity} holds for arbitrary quadratic forms in arbitrary characteristic. Indeed, this may be deduced as an easy consequence of the fundamental ``representation theorem'' of Cassels and Pfister (\cite[Thm. 17.3]{EKM}). For $m>1$, however, the statement is peculiar to the quasilinear case; as observed by Hoffmann (\cite[Cor. 3.4]{Hoffmann2}), one may use here the additivity of quasilinear quadratic forms to generalize the Cassels-Pfister theorem to the following stronger multi-variable statement:

\begin{theorem}[{Hoffmann, cf. \cite[Cor. 3.4]{Hoffmann2}}] \label{thmcprep} Let $q$ be a quasilinear quadratic form over $F$, let $T = (T_1,\hdots,T_n)$ be a system of algebraically independent variables over $F$ and let $f \in F[T]$. If $f \in D(q_{F(T)})$, then $f \in D(q_{F[T]})$. \end{theorem}

\section{Main results: The quasilinear case}

In this final section we prove several new results concerning the isotropy behaviour of quasilinear quadratic forms over function fields of (totally singular) quadrics. The main result is Theorem \ref{thmsubformdec}, from which a number of interesting statements, including the quasilinear part of Theorem \ref{thmmain}, follow formally. This theorem is essentially a generalization of \cite[Thm. 5.1]{Scully3}, and its proof follows that of the latter closely. However, certain adjustments are needed along the way to facilitate the extra generality. Among them, we need to work with certain ``generic subforms'' of a given anisotropic form. We therefore begin with a brief discussion around this idea. As in the previous section, $F$ will denote an arbitrary field of characteristic $2$ throughout.

\subsection{Generic subforms} The following lemma is well known, though we do not have a precise reference for it:

\begin{lemma} \label{lemspecofrationalpoints}  Let $k$ be a field and let $f \colon X \rightarrow Y$ be a proper morphism of $k$-schemes with $Y$ irreducible. If the generic fibre of $f$ has a rational point, then every fibre of $f$ over the regular locus of $Y$ has a rational point.

\begin{remark} Here we are regarding the fibre over a point $y \in Y$ as a scheme over the residue field $k(y)$. In particular, a rational point of $f^{-1}(y)$ means a $k(y)$-point of $f^{-1}(y)$. \end{remark}

\begin{proof} Assume that the generic fibre of $f$ has a rational point, and let $y$ be a regular point of $Y$ with reduced closure $Z$ in $Y$. If $Z$ has codimension $1$ in $Y$, then $\mathcal{O}_{Y,y}$ is a discrete valuation ring, and the statement follows from \cite[(7.3.8)]{EGAII} Otherwise, the regularity of $\mathcal{O}_{Y,y}$ implies that $y$ is a regular point of some codimension-1 integral closed subvariety of $Y$. The statement therefore follows by induction on the codimension of $Z$ in $Y$. \end{proof} \end{lemma}

Using this, we can prove the following existence statement which will be needed for the proof of Theorem \ref{thmsubformdec} below:

\begin{lemma} \label{lemgenericsubforms} Let $p$ be an anisotropic quasilinear quadratic form of dimension $\geq 2$ over $F$ such that $1 \in D(p)$. Then there exist a purely transcendental field extension $L$ of $F$ and a codimension-1 subform $p' \subset p_L$ such that
\begin{enumerate} \item $L[p']$ is a purely transcendental extension of $F(p)$.
\item $p_L \simeq \form{1} \perp p'$. \end{enumerate}

\begin{remark} We remind the reader that $L[p']$ denotes here the function field of the \emph{affine} (as opposed to projective) hypersurface defined by the vanishing of $p'$ over $L$. \end{remark}

\begin{proof} We construct the generic such $p'$ (compare \cite[Lem. 4.2]{Totaro1}) Let $q = \form{1} \perp p$ with underlying vector space $\ell \oplus V_p$, and let $P \subset \mathbb{P}(V_p)$ and $Q \subset \mathbb{P}(\ell \oplus V_p)$ denote the projective $F$-quadrics defined by the vanishing of $p$ and $q$, respectively. Let $Y$ denote the Grassmannian of hyperplanes in $V_p$ and consider the closed subschemes 
$$ I_1 = \lbrace (x,H) \in P \times Y\;|\;x \in \mathbb{P}(H) \rbrace$$
and
$$ I_2 = \lbrace (x,H) \in Q \times Y\;|\;x \in \mathbb{P}(\ell \oplus H) \rbrace$$
of $P \times Y$ and $Q \times Y$, respectively. Then the generic fibre $P'$ of the canonical projection $I_1 \rightarrow Y$ is the projective quadric defined by the vanishing of a codimension-1 subform $p' \subset p_L$, where $L = F(Y)$. Since $Y$ is a rational $F$-variety, $L$ is a purely transcendental extension of $F$. We claim that $p'$ satisfies conditions (1) and (2). For (1), it suffices to show that $L(p')$ is a purely transcendental extension of $F(p)$ ($L[p']$ is a purely transcendental extension of $L(p')$ -- see \ref{secfunctionfields}). Note, however, that $L(p')$ is also canonically $F$-isomorphic to the function field of the generic fibre of the projection $I_1 \rightarrow P$. As the latter fibre is nothing else but the projective space of hyperplanes in $\mathbb{P}(V_p \otimes_F F(p)\big)$ containing the canonical rational point of $P \times_F \mathrm{Spec}\big(F(p)\big)$, (1) follows. For (2), note that since $p$ is a quasilinear quadratic form representing $1$, we have $D(\form{1} \perp p') \subset D(p_L)$. As $p_L$ is anisotropic, Proposition \ref{propsubforms} implies that $\anispart{(\form{1} \perp p')} \subset p_L$, so by dimension reasons it will be enough to check that $\form{1} \perp p'$ is anisotropic. But the vanishing locus of $\form{1} \perp p'$ is nothing else but the generic fibre of the canonical projection from the second incidence scheme $I_2$ to $Y$. Since the latter projection is (i) proper (as $Q$ is projective over $F$) and (ii) has fibres over regular points which do not posses a rational point (because $p$ is anisotropic, there exists a hyperplane of $V_p$ where $p$ does not represent $1$), its generic fibre admits no rational point by Lemma \ref{lemspecofrationalpoints}. In other words, $\form{1} \perp p'$ is anisotropic, as claimed. \end{proof} \end{lemma}

\begin{remark} \label{remtot} In the situation of Lemma \ref{lemgenericsubforms}, let $\widehat{p} \subset p$ be such that $p \simeq \form{1} \perp \widehat{p}$. If $\witti{1}{p} > 1$, then it has been shown by Totaro (\cite[Thm. 6.4]{Totaro1}) that assertion (1) holds with $L = F$ and $p' = \widehat{p}$, i.e., that $F[\widehat{p}]$ is $F$-isomorphic to $F(p)$ (see also \cite[Cor. 3.9]{Scully3}). We will not use this fact in what follows.\end{remark}

\subsection{A subform condition for isotropy of quasilinear quadratic forms over the function field of a quadric}

The following theorem is the main result of this section; we remark that if $\witti{1}{p}>1$, then the ``up to replacing $F$ with a purely transcendental extension of itself'' qualification may be removed -- see Remark \ref{remquasilinearmain} below:

\begin{theorem} \label{thmsubformdec} Let $p$ and $q$ be anisotropic quasilinear quadratic forms of dimension $\geq 2$ over $F$. If $q_{F(p)}$ is isotropic, then, up to replacing $F$ with a purely transcendental extension of itself, there exists an anisotropic quasilinear quadratic form $\tau$ over $F(p)$ such that
\begin{enumerate} \item $\mathrm{dim}(\tau) = \witti{0}{q_{F(p)}}$, and
\item $\anispart{(p_1 \otimes \tau)} \subset \anispart{(q_{F(p)})}$. \end{enumerate} \end{theorem}

\begin{proof} After multiplying $q$ and $p$ by appropriate scalars if necessary, we may assume that both forms represent $1$. By Lemma \ref{lemgenericsubforms}, we can then find a purely transcendental field extension $L$ of $F$ and a subform $p' \subset p_L$ such that $p_L \simeq \form{1} \perp p'$ and such that the affine function field $L[p']$ is a purely transcendental extension of $F(p)$. We fix such a pair $(L,p')$ for the remainder of the proof. Now, let us choose elements $a_1,\hdots,a_n \in L^*$ such that $p' \simeq \form{a_1,\hdots,a_n}$, and set $p'(T) = \sum_{i=1}^n a_iT_i^2 \in L[T]$, where $T = (T_1,\hdots,T_n)$ is a system of algebraically independent variables over $L$. Then the field $L[p']$ may be identified with the fraction field of the integral domain $L[T]/\big(p'(T)\big)$ (see \S \ref{secfunctionfields} above). Fixing this identification henceforth, we will write $\overline{f}$ for the image of a given polynomial $f \in L[T]$ under the canonical projection $L[T] \rightarrow L[p']$. For such a polynomial $f$, we will also write $m(f)$ for the multiplicity of $p'(T)$ in $f$, i.e., the largest integer $k$ such that $f = p'(T)^kh$ for some $h \in L[T]$. Note that we have $m(f) = 0$ if and only if $\overline{f} \neq 0$ in $L[p']$. \newline

To simplify the notation, we now let $\mathfrak{i} = \witti{0}{q_{F(p)}}$. By hypothesis, we have $\mathfrak{i} > 0$. The proof of Theorem \ref{thmsubformdec} begins with the following lemma:

\begin{lemma} \label{lempresentationlemma} Assume we are in the above situation. Then there exists a subform $r \subset q$ and elements $f_1,\hdots,f_\mathfrak{i} \in D(r_{L[T]})$ such that:
\begin{enumerate} \item $r_{F(p)} \simeq \anispart{(q_{F(p)})}$.
\item $r_{L[p']}$ is anisotropic.
\item $q_{L(T)} \simeq r_{L(T)} \perp p'(T)\form{f_1,\hdots,f_{\mathfrak{i}}}$.
\item The form $\tau = \form{\overline{f_1},\hdots,\overline{f_\mathfrak{i}}}$ $($defined over $L[p'])$ is anisotropic.
\end{enumerate} 
\begin{proof} The existence of a subform $r \subset q$ satisfying (1) is ensured by Corollary \ref{corexcellence}. 
Since $L[p']$ is a purely transcendental extension of $F(p)$, any such $r$ also satisfies (2) in view of Lemma \ref{lemseparable}. Now, as $L$ is a purely transcendental extension of $F$, the field $L[p]$ is a purely transcendental extension of $F[p]$, and hence of $F(p)$. Another application of Lemma \ref{lemseparable} therefore shows that $r_{L[p]} \simeq \anispart{(q_{L[p]})}$. Since $L[p]$ is $L$-isomorphic to $L(T)\left(\sqrt{p'(T)}\right)$ (see \S \ref{secfunctionfields} above), Lemma \ref{lemquadraticextensions} then implies that
$$ q_{L(T)} \simeq r_{L(T)} \perp p'(T)\form{f_1,\hdots,f_\mathfrak{i}} $$
for some $f_1,\hdots,f_\mathfrak{i} \in D(r_{L(T)})$. Multiplying the $f_i$ by squares in $L[T]$ if necessary, we may arrange it so that $f_i \in D(r_{L[T]})$ for all $i$. Thus, there exists a sequence $f_1,\hdots,f_\mathfrak{i} \in D(r_{L[T]})$ for which (3) holds. Among all such sequences, let us choose one for which $\sum_{i=1}^{\mathfrak{i}} \mathrm{deg}_{T_1}f_i$ is \emph{minimal} \big(where, for a non-zero polynomial $f \in L[T]$, $\mathrm{deg}_{T_1}(f)$ denotes the degree of $f$ viewed as a polynomial in the variable $T_1$ over the field $L(T_2,\hdots,T_n)$\big). We claim that for this choice of $f_1,\hdots,f_\mathfrak{i}$, (4) also holds. Suppose, for the sake of contradiction, that this is not the case, i.e., that the form $\form{\overline{f_1},\hdots,\overline{f_\mathfrak{i}}}$ is isotropic over $L[p']$. Then there exist polynomials $h_1,\hdots,h_\mathfrak{i},h \in L[T]$ such that
\begin{enumerate} \item[(i)] $\sum_{i=1}^n h_i^2f_i = p'(T)h$ in $L[T]$.
\item[(ii)] $\mathrm{deg}_{T_1}(h_i) < 2$ for all $i$.
\item[(iii)] $\overline{h_i} \neq 0$ for at least one $i$. \end{enumerate}
Indeed, (i) and (iii) amount to the stated isotropy of the form, while (ii) can be arranged because $T_1^2 = \sum_{i=2}^n \frac{a_i}{a_1}T_i^2$ in the field $L[p']$. Now, as $D(r_{L(T)})$ is closed under addition, we have $p'(T)h \in D(r_{L(T)})$ by (i) and the choice of the $f_i$. Since $r_{L[p']}$ is anisotropic, Lemma \ref{lemmultiplicity} implies that $h = p'(T)h'$ for some $h' \in L[T]$. Note here that $h' \neq 0$. Indeed, if this were not the case, then (i) and (iii) would imply that $\form{f_1,\hdots,f_{\mathfrak{i}}}$ is isotropic over $L(T)$; since the latter form is similar to a subform of the anisotropic form $q_{L(T)}$, this is not so. Furthermore, since $p'(T)^2h' = p'(T)h \in D(r_{L(T)})$, and since $D(r_{L(T)})$ is closed under multiplication by squares in $L(T)$, we have $h' \in D(r_{L(T)})$. Taking the generalized Cassels-Pfister representation theorem for quasilinear quadratic forms into account (Theorem \ref{thmcprep}), we see that, in fact, $h' \in D(r_{L[T]})$. Now, by (iii), there exists an integer $l \in \lbrace 1,\hdots, \mathfrak{i} \rbrace$ such that $h_l \neq 0$. Among all such $l$, let us choose one so that the integer $\mathrm{deg}_{T_1}(h_l^2f_l)$ is maximal. Since $h_l,h' \neq 0$, Proposition \ref{propsubforms} \big(together with (i)\big) implies that $\form{f_1,\hdots,f_\mathfrak{i}} \simeq \form{f_1,\hdots,f_{l-1},p'(T)^2h',f_{l+1},\hdots,f_\mathfrak{i}} \simeq \form{f_1,\hdots,f_{l-1},h',f_{l+1},\hdots,f_\mathfrak{i}}$ as forms over $L(T)$. In other words, $f_1,\hdots,f_{l-1},h',f_{l+1},\hdots,f_\mathfrak{i}$ is a sequence of non-zero elements in $D(r_{L[T]})$ satisfying condition (3). But, since $\mathrm{deg}_{T_1}\big(p'(T)\big) = 2$, (i), (ii) and the choice of $l$ imply that $\mathrm{deg}_{T_1}(h') < \mathrm{deg}_{T_1}(f_l)$, and this contradicts our original choice of the $f_i$. We conclude that our supposition was incorrect, and so the lemma is proved. \end{proof} \end{lemma} 

Let us now fix a subform $r \subset q$ and elements $f_1,\hdots,f_\mathfrak{i} \in D(r_{L[T]})$ satisfying the four conditions of Lemma \ref{lempresentationlemma}. Since $L[p']$ is a purely transcendental extension of $F(p)$, Theorem \ref{thmsubformdec} now follows from the following more precise lemma:

\begin{lemma} \label{lemkeydiv} Let $K = L[p']$. Then, in the above situation, we have
$$ \anispart{\big((p_1)_K \otimes \form{\overline{f_1},\hdots,\overline{f_\mathfrak{i}}}\big)} \subset \anispart{(q_K)} $$
\end{lemma}
\begin{proof} By Proposition \ref{propsubforms}, the statement is equivalent to the assertion that $$ D\big((p_1)_K\otimes \form{\overline{f_1},\hdots,\overline{f_\mathfrak{i}}}\big) \subset D(q_K).$$
As both sides are $K^2$-vector spaces, it suffices to show that the right-hand side contains a set of generators for the left-hand side. Since $D\big((p_1)_K\big)$ is generated by $D(p_L)$ over $K^2$, it is therefore sufficient to show that $b\overline{f_i} \in D(q_K)$ for all $b \in D(p_L)$ and all $1 \leq i \leq \mathfrak{i}$. In fact, it suffices to show this in the case where $b \in D(p_L) \setminus D(p')$. Indeed, if $b \in D(p')$, then $1+b \in D(p_L) \setminus D(p')$; since $D(q_K)$ is a $K^2$-vector space which (by construction) contains $\overline{f_i}$, we have $b\overline{f_i} \in D(q_K) \Leftrightarrow (1+b)\overline{f_i} \in D(q_K)$. Now, in order to check that the statement holds, we first need the following preliminary calculation:

\begin{sublemma} \label{sublemmultiplier} Assume we are in the above situation, and let $b \in D(p_L)\setminus D(p')$ and $1 \leq i \leq \mathfrak{i}$. Then there exist $s_{i,j} \in D(r_{L[T]})$ and $t_{i,j} \in L[T] \setminus \lbrace 0 \rbrace$ $(0 \leq j \leq 1)$ such that:
\begin{enumerate} \item $bf_i = \frac{s_{i,0}}{t_{i,0}^2} + \frac{s_{i,1}}{t_{i,1}^2} \frac{p'(T)}{b}$ in $L(T)$.
\item For each $j \in \lbrace 0,1 \rbrace$, at least one of $\overline{s_{i,j}}$ and $\overline{t_{i,j}}$ is a non-zero element of $K=L[p']$.\end{enumerate}

\begin{proof} For simplicity of notation, let $f = f_i$. Consider the field $M = L(T)\left(\sqrt{\frac{p'(T)}{b}}\right)$. By \S \ref{secfunctionfields}, $M$ may be identified with the function field $L[\eta]$ of the affine quadric over $L$ defined by the vanishing of $\eta = \form{b} \perp p'$. Since $p_L = \form{1} \perp p'$, and since $b \in D(p_L) \setminus D(p')$, we have $D(\eta) = D(p_L)$. In view of Proposition \ref{propsubforms}, it follows that $\eta \simeq p_L$. In particular, $M$ is a degree-1 purely transcendental extension of $L(p)$ (see \S \ref{secfunctionfields}), and is thus a purely transcendental extension of $F(p)$. By Lemma \ref{lemseparable} and the choice of $r$, it follows that $r_M \simeq \anispart{(q_M)}$. Again, by Proposition \ref{propsubforms}, this means that $D(q_M) = D(r_M)$. Now, since $u = \frac{p'(T)}{b}$ is a square in $M$, we have $bf= \frac{p'(T)f}{u} \in D(q_M)$, and so $bf \in D(r_M)$. Since
$$D(r_M) = D(r_{L(T)}) + uD(r_{L(T)})=  D(r_{L(T)}) + \frac{p'(T)}{b}D(r_{L(T)})$$ 
as a subset of $L(T)$ (see Lemma \ref{lemvaluesquadratic}), it follows that we can write $bf = q_0 + q_1 \frac{p'(T)}{b}$ for some $q_0,q_1 \in D(r_{L(T)})$. Because every element of $D(r_{L(T)})$ is evidently the ratio of an element of $D(r_{L[T]})$ and a non-zero square in $L[T]$, we can therefore find elements $s_j \in D(r_{L[T]})$ and $t_j \in L[T] \setminus \lbrace 0 \rbrace$ ($0 \leq j \leq 1$) so that
\begin{equation} \label{eq1} bf = \frac{s_0}{t_0^2} + \frac{s_1}{t_1^2} \frac{p'(T)}{b} \end{equation}
in $L(T)$. Now, for each $j \in \lbrace 0,1 \rbrace$, let $m_j = \mathrm{min}\big(m(s_j),2m(t_j)\big)$ (with $m(s_j)$, $m(t_j)$ defined as in the beginning of the proof). Since $r_K$ is anisotropic, Lemma \ref{lemmultiplicity} shows that the $m(s_j)$, and hence the $m_j$, are even. In particular, if we let
$$ s_j' = \frac{s_j}{p'(T)^{m_j}} \hspace{.5cm} \text{and} \hspace{.5cm} t_j' = \frac{t_j}{p'(T)^{m_j/2}},$$ 
then $s_j' \in D(r_{L[T]})$ and $t_j' \in L[T]\setminus \lbrace 0 \rbrace$ for each $j$. Replacing the pair $(s_0,t_0)$ with $(s_0',t_0')$ and the pair $(s_1,t_1)$ with $(s_1',t_1')$ (this does not alter \eqref{eq1}), we arrive at the situation where each of the pairs $(\overline{s_0},\overline{t_0})$ and $(\overline{s_1},\overline{t_1})$ has at least one non-zero entry, as we wanted. \end{proof} \end{sublemma} 

Returning to the proof of Lemma \ref{lemkeydiv}, let $b \in D(p_L)\setminus D(p')$ and let $1 \leq i \leq \mathfrak{i}$. By Sublemma \ref{sublemmultiplier}, we have
$$ bf_i = \frac{s_{i,0}}{t_{i,0}^2} + \frac{s_{i,1}}{t_{i,1}^2} \frac{p'(T)}{b}$$
for some $s_{i,j} \in D(r_{L[T]})$ and $t_{i,j} \in L[T] \setminus \lbrace 0 \rbrace$ $(0 \leq j \leq 1)$ such that each of the pairs $(\overline{s_{i,0}},\overline{t_{i,0}})$ and $(\overline{s_{i,1}},\overline{t_{i,1}})$ has at least one non-zero entry. We claim that both $\overline{t_{i,0}}$ and $\overline{t_{i,1}}$ are non-zero, or, equivalently, that $m(t_{i,j}) = 0$ for each $j$. To see this, let us first clear denominators in the preceding equation to obtain the equality
\begin{equation} \label{eq2} bf_it_{i,0}^2t_{i,1}^2 = s_{i,0}t_{i,1}^2 + s_{i,1}t_{i,0}^2 \frac{p'(T)}{b}\end{equation}
in the polynomial ring $L[T]$. Now, let $m = \mathrm{min}\big(m(t_{i,0}),m(t_{i,1})\big)$. Our claim is then equivalent to the assertion that $m = m(t_{i,0}) + m(t_{i,1})$. Suppose that this is not the case, and let $j \in \lbrace 0 ,1 \rbrace$ be minimal so that $m(t_{i,l}) = m$, where $l$ is the integer complementary to $j$ in $\lbrace 0,1\rbrace$. Then, reducing both sides of \eqref{eq2} modulo $p'(T)^{2m + j + 1}$, we see that $s_{i,j} \equiv 0 \pmod{p'(T)}$, i.e., that $\overline{s_{i,j}} = 0$. Since the pair $(\overline{s_{i,j}},\overline{t_{i,j}})$ has at least one non-zero entry, it follows that $\overline{t_{i,j}} \neq 0$, or, equivalently, that $m(t_{i,j}) = 0$. But then $m = m(t_{i,0}) + m(t_{i,1})$, which contradicts our supposition. The claim is therefore valid, and so, reducing \eqref{eq2} modulo $p'(T)$ and dividing through by $(\overline{t_{1,0}}\overline{t_{i,1}})^2$, we obtain that
$$ b\overline{f_i} = \overline{s_{i,0}}/\overline{t_{i,0}}^2 $$
in $K$. As $s_{i,0} \in D(r_{L[T]}) \subset D(q_{L[T]})$, this shows that $ b\overline{f_i} \in D(q_K)$, as we needed. \end{proof}

As per the above discussion, Theorem \ref{thmsubformdec} is now proved. \end{proof}

\begin{remark} \label{remquasilinearmain} As already mentioned, if $\witti{1}{p} > 1$, then ``up to replacing $F$ with a purely transcendental extension of itself'' may be removed from the statement of Theorem \ref{thmsubformdec}. Indeed, if $p \simeq \form{1} \perp \widehat{p}$ and $\witti{1}{p}>1$, then we can choose $(F,\widehat{p})$ for the pair $(L,p')$ which was used throughout the proof -- see Remark \ref{remtot} above. It is unclear to the author whether the qualification is really needed when $\witti{1}{p}=1$. In any case, the statement which we have proved is sufficient for the basic applications. \end{remark}

\subsection{Applications} \label{secquasilinearapplications}

We now provide some concrete applications of Theorem \ref{thmsubformdec} to the problem of understanding the splitting behaviour of quasilinear quadratic forms under scalar extension to the function field of a quadric. In particular, we will prove the quasilinear case of Theorem \ref{thmmain}. The reader will note that, in this case, the form $p_1$ plays a similar role to that played by the upper motive of the quadric $P$ in the proof of the characteristic $\neq 2$ case given in \S \ref{secchar0} above. It will be of particular importance to remember that $p_1$ has dimension $\mathrm{dim}(p) - \witti{1}{p}$ in the quasilinear setting, as opposed to the more familiar $\mathrm{dim}(p) -2\witti{1}{p}$; indeed, if $\phi$ is any non-zero quasilinear quadratic form over $F$, then $\mathrm{dim}(\anispart{\phi}) = \mathrm{dim}(\phi) - \witti{0}{\phi}$ (see \S \ref{secisodecomp} above). The ``divisibility indices'' introduced in \S \ref{secdivbyqp} will also have a key role to play here; in effect, the study of these indices replaces the (implicit) use of mod-2 Steenrod operations on Chow groups in the characteristic $\neq 2$ setting. Before proceeding, we make a general remark:

\begin{remark} \label{remgenericsmoothness} Let $p$ and $q$ be anisotropic quadratic forms of dimension $\geq 2$ over $F$ with $q$ quasilinear. If $p$ is \emph{not} quasilinear, then its associated quadric is generically smooth, which amounts to the assertion that its function field $F(p)$ is a separable extension of $F$ (see \cite[(17.15.9)]{EGAIV}). In view of Lemma \ref{lemspecofrationalpoints}, we therefore have $\witti{0}{q_{F(p)}} = 0$ in this case. As a result, when studying the isotropy behaviour of the quasilinear form $q$ under scalar extension to the field $F(p)$, the only case of interest is that where $p$ is also quasilinear. \end{remark}

The first interesting application of Theorem \ref{thmsubformdec} is Theorem \ref{thmp1subform} below, from which one obtains a short proof of the quasilinear part of Theorem \ref{thmKMT}. Before proving Theorem \ref{thmp1subform}, we first make the following quick observation:

\begin{lemma} \label{lemsubforminsensitivity} Let $\psi$ and $\phi$ be anisotropic quasilinear quadratic forms over $F$, and let $L$ be a purely transcendental field extension of $F$. If $\psi_L$ is similar to a subform of $\phi_L$, then $\psi$ is similar to a subform of $\phi$.
\begin{proof} If $F$ is finite, then $F$ is perfect and so $\psi$ and $\phi$ are necessarily split (i.e., 1-dimensional). Since the statement is trivial in this case, we may assume that $F$ is infinite. Furthermore, the problem can be easily reduced to the case where $L$ has finite transcendence degree over $F$ using Proposition \ref{propsubforms}. We can therefore also assume that $L = F(T)$, where $T = (T_1,\hdots,T_n)$ is a system of algebraically independent variables over $F$. Now, if $\psi_L$ is similar to a subform of $\phi_L$, then we can evidently find a polynomial $f(T) \in F[T]$ such that $f(T)\psi_L \subset \phi_L$. Since $F$ is infinite, there exist scalars $a_1,\hdots,a_n \in F$ such that $f(a_1,\hdots,a_n) \neq 0$. Letting $a = f(a_1,\hdots,a_n)$, we now claim that $a\psi \subset \phi$. By Proposition \ref{propsubforms}, it suffices to show that $ab \in D(\phi)$ for all $b \in D(\psi)$. Since $f(T)\psi_L \subset \phi_L$, we certainly have $f(T)b \in D(\phi_L)$. By the generalized Cassels-Pfister representation theorem for quasilinear forms (Theorem \ref{thmcprep}), it follows that $f(T)b \in D(\phi_{F[T]})$. The claim then follows by performing the specialization $(T_1,\hdots,T_n) \dashrightarrow (a_1,\hdots,a_n)$, which is well defined on the polynomial ring $F[T]$.\end{proof} \end{lemma}

\begin{remark} In fact, the quasilinearity hypothesis is not necessary here; one may show that the statement holds for an arbitrary pair of anisotropic quadratic forms over a field of any characteristic using induction on the transcendence degree of $L$ and the general representation theorems of Pfister (\cite[Thm. 17.12]{EKM}) and Cassels-Pfister (\cite[Thm. 17.3]{EKM}). We refrain from going into the details here. \end{remark}

\begin{theorem} \label{thmp1subform} Let $p$ and $q$ be anisotropic quasilinear quadratic forms of dimension $\geq 2$ over $F$. If $q_{F(p)}$ is isotropic, then $p_1$ is similar to a subform of $\anispart{(q_{F(p)})}$.
\begin{proof} By Lemma \ref{lemsubforminsensitivity}, it is sufficient to show this after replacing $F$ with a purely transcendental extension of itself. By Theorem \ref{thmsubformdec}, we may therefore assume that there exists an anisotropic form $\tau$ of dimension $\witti{0}{q_{F(p)}} \geq 1$ over $F(p)$ such that $\anispart{(p_1 \otimes \tau)} \subset \anispart{(q_{F(p)})}$. If $a \in D(\tau)$, then $D(ap_1) \subseteq D(p_1 \otimes \tau) \subset D(q_{F(p)})$. Since $p_1$ is anisotropic, Proposition \ref{propsubforms} then implies that $ap_1 \subset \anispart{(q_{F(p)})}$, which proves the theorem. \end{proof} \end{theorem}

As a consequence of Theorem \ref{thmp1subform}, we get an elegant explanation for the quasilinear part of Theorem \ref{thmKMT}. This result was originally proved by Totaro in \cite{Totaro1} using the basic machinery of Chow groups (see [\emph{loc. cit.}, Thm. 5.2]); a more elementary proof (rather different to the one presented here) was later given in \cite{Scully1}:

\begin{corollary}[Totaro] \label{corKMTquasilinear} Let $p$ and $q$ be anisotropic quasilinear quadratic forms of dimension $\geq 2$ over $F$. If $q_{F(p)}$ is isotropic, then $\witti{0}{q_{F(p)}} \leq \mathrm{dim}(q) - \mathrm{dim}(p) + \witti{1}{p}$. 
\begin{proof} By Theorem \ref{thmp1subform}, $p_1$ is similar to a subform of $\anispart{(q_{F(p)})}$. In particular, we have
$$\mathrm{dim}(p) -\witti{1}{p} = \mathrm{dim}(p_1) \leq \mathrm{dim}(\anispart{(q_{F(p)})}) = \mathrm{dim}(q) - \witti{0}{q_{F(p}}. $$
Rearranging this inequality, we obtain the desired assertion. \end{proof} \end{corollary}

The next applications make use of the fact that the form $\tau$ appearing in the statement of Theorem \ref{thmsubformdec} has dimension at least $\witti{0}{q_{F(p)}}$. We begin by re-deriving the main result of \cite{Scully3}. This result was used in [\emph{loc. cit.}, \S 6.1] to determine all possible values of the Knebusch splitting pattern for quasilinear forms. This includes the quasilinear part of Theorem \ref{thmKarpenko}, which is an immediate dimension-theoretic consequence:

\begin{corollary}[{\cite[Thm. 5.1]{Scully3}}] \label{corp1div} Let $p$ be an anisotropic quasilinear quadratic form of dimension $\geq 2$ over $F$. Then $2^{\mathfrak{d}_1(p)} \geq \witti{1}{p}$. In other words, $p_1$ is divisible by a quasi-Pfister form of dimension $\geq \witti{1}{p}$. 
\begin{proof} By Lemma \ref{lemdivinvariance}, the statement is insensitive to replacing $F$ with a purely transcendental extension of itself. Applying Theorem \ref{thmsubformdec} to the case where $p=q$, we can therefore assume that we have an anisotropic form $\tau$ of dimension $\witti{1}{p}$ over $F(p)$ such that $\anispart{(p_1 \otimes \tau)} \subset p_1$. As per the proof of Theorem \ref{thmp1subform}, however, $p_1$ is similar to a subform of $\anispart{(p_1 \otimes \tau)}$. For dimension reasons, we therefore conclude that $p_1 \simeq \anispart{(p_1 \otimes \tau)}$. By Corollary \ref{cordivbynormform}, this implies that $p_1$ is divisible by the quasi-Pfister form $\normform{\tau}$. But, since $\tau$ is anisotropic, $\tau$ is similar to a subform of $\normform{\tau}$ (see \S \ref{secnormform}). In particular, we have $\mathrm{dim}(\normform{\tau}) \geq \mathrm{dim}(\tau) = \witti{1}{p}$, and this completes the proof. \end{proof} \end{corollary}

Given Corollary \ref{corp1div}, we can now generalize it as follows:

\begin{theorem} \label{thmquasilinearmainapplication} Let $p$ and $q$ be anisotropic quasilinear quadratic forms of dimension $\geq 2$ over $F$ such that $q_{F(p)}$ is isotropic. If $\witti{0}{q_{F(p)}} > \mathrm{dim}(q) - \mathrm{dim}(p)$, then $2^{\mathfrak{d}_1(p)} \geq \witti{0}{q_{F(p)}}$. In other words, $p_1$ is divisible by a quasi-Pfister form of dimension $\geq \witti{0}{q_{F(p)}}$ under the given hypotheses.
\begin{proof} Again, by Lemma \ref{lemdivinvariance}, the statement is insensitive to replacing $F$ with a purely transcendental extension of itself. Thus, by Theorem \ref{thmsubformdec}, we may assume that there exists an anisotropic form $\tau$ of dimension $\witti{0}{q_{F(p)}}$ over $F(p)$ such that $\anispart{(p_1 \otimes \tau)} \subset \anispart{(q_{F(p)})}$. Let $\eta = \anispart{(p_1 \otimes \tau)}$. As in the proof of Corollary \ref{corp1div}, $p_1$ is similar to a subform of $\eta$, and the theorem will follow (by exactly the same arguments) if we can show that the two forms have the same dimension under the given hypotheses. Suppose, for the sake of contradiction, that this is not the case. Then $\mathrm{dim}(\eta) \geq \mathrm{dim}(p_1) + 2^{\mathfrak{d}_1(p)}$. Indeed, since $p_1$ is, by definition, divisible by an anisotropic quasi-Pfister form of dimension $2^{\mathfrak{d}_1(p)}$, Lemma \ref{lemanisdivisibility} implies that the same is true of $\eta$. Thus, both $\mathrm{dim}(p_1)$ and $\mathrm{dim}(\eta)$ are divisible by $2^{\mathfrak{d}_1(p)}$, and so our claim follows. In particular, since $\eta \subset \anispart{(q_{F(p)})}$, we have
$$ \mathrm{dim}(p) - \witti{1}{p} + 2^{\mathfrak{d}_1(p)} = \mathrm{dim}(p_1) +  2^{\mathfrak{d}_1(p)} \leq \mathrm{dim}(\eta) \leq \mathrm{dim}(\anispart{(q_{F(p)})}) = \mathrm{dim}(q) - \witti{0}{q_{F(p)}}. $$
Now Corollary \ref{corp1div} asserts that $2^{\mathfrak{d}_1(p)} \geq \witti{1}{p}$. Together with the previous inequality, this gives
$$ \witti{0}{q_{F(p)}} \leq \mathrm{dim}(q) - \mathrm{dim}(p), $$
which contradicts our original hypothesis and thus completes the proof of the theorem. \end{proof} \end{theorem}

Taking dimensions, we obtain the quasilinear part of Theorem \ref{thmmain} (in light of Remark \ref{remgenericsmoothness}, it is enough to treat the case where both $p$ and $q$ are quasilinear):

\begin{corollary} \label{cormainboundquasilinear} Let $p$ and $q$ be anisotropic quasilinear quadratic forms of dimension $\geq 2$ over $F$. Then $ \witti{0}{q_{F(p)}} \leq \mathrm{max}\big(\mathrm{dim}(q) -\mathrm{dim}(p),2^{\mathfrak{d}_1(p)}\big)$. In particular, setting\\ $s = v_2\big(\mathrm{dim}(p) - \witti{1}{p}\big)$, we have
$$ \witti{0}{q_{F(p)}} \leq \mathrm{max}\big(\mathrm{dim}(q) -\mathrm{dim}(p),2^{s}\big).$$
\begin{proof} The first statement follows immediately from Theorem \ref{thmquasilinearmainapplication}. The second follows from the first, since $2^{\mathfrak{d}_1(p)}$ is, by definition, a divisor of $\mathrm{dim}(p_1) = \mathrm{dim}(p) - \witti{1}{p}$. \end{proof} \end{corollary}

\begin{remark} In the situation of Corollary \ref{cormainboundquasilinear}, the inequality $\mathfrak{d}_1(p) \leq s$ need not be an equality. For example, if $p$ is ``generic'' of dimension $2^n + 1$ for some positive integer $n$, then $\mathfrak{d}_1(p) = 0$ (see \cite[Lem. 2.46]{Scully3}), while $s = n$ by Theorem \ref{thmKarpenko}. Interestingly, we do not know of an analogue of the integer $\mathfrak{d}_1(p)$ in the characteristic $\neq 2$ theory. \end{remark}

In fact, the proof of Theorem \ref{thmquasilinearmainapplication} shows that Corollary \ref{cormainboundquasilinear} may be refined as follows:

\begin{corollary} \label{corquasilinearmainrefined} Let $p$ and $q$ be anisotropic quasilinear quadratic forms of dimension $\geq 2$ over $F$. Then $2^{\mathfrak{d}_1(p)} \geq \witti{1}{p}$ and
$$ \witti{0}{q_{F(p)}} \leq \mathrm{max}\big(\mathrm{dim}(q) -\mathrm{dim}(p)+\witti{1}{p} - 2^{\mathfrak{d}_1(p)},2^{\mathfrak{d}_1(p)}\big).$$ \end{corollary}

\begin{remark} This refinement is non-trivial if $2^{\mathfrak{d}_1(p)} > \witti{1}{p}$, which happens, for example, if $p$ is odd-dimensional and $\witti{1}{p} > 1$. Indeed, in that case, $\witti{1}{p}$ is odd by Corollary \ref{corp1div}. \end{remark}

Corollary \ref{cormainboundquasilinear} above was deduced from the statement of Theorem \ref{thmsubformdec} by identifying a certain situation in which the subform $\anispart{(p_1 \otimes \tau)}$ of $\anispart{(q_{F(p)})}$ is similar to $p_1$. Another interesting problem is to find circumstances under which $\anispart{(p_1 \otimes \tau)}$ is similar to $\tau$. As the proof of the next proposition shows, one situation in which this happens is that where $\witti{0}{q_{F(p)}}$ is ``close'' to attaining its maximal possible value of $\mathrm{dim}(q)/2$. Unfortunately, formulating what this means in general terms necessitates a certain degree of technicality; we hope, however, that the subsequent discussion will help to illuminate the more concrete meaning of our observation. Before stating the proposition, we recall (see \S \ref{secnormform}) that if $p$ is an anisotropic quasilinear form over $F$, then $\mathrm{lndeg}(p)$ denotes the integer $\mathrm{log}_2\big(\mathrm{dim}(\normform{p})\big)$. Since $p$ is similar to a subform of $\normform{p}$, we have $\mathrm{dim}(p) \leq 2^{\mathrm{lndeg}(p)}$. 

\begin{proposition} \label{propquasilinearsecondmainapplication} Let $p$ and $q$ be anisotropic quasilinear quadratic forms of dimension $\geq 2$ over $F$ and let $\epsilon$ be the unique integer in $[1,2^{\mathfrak{d}_1(p)}]$ such that $\witti{0}{q_{F(p)}} + \epsilon \equiv 0 \pmod{2^{\mathfrak{d}_1(p)}}$. If $\mathrm{dim}(q) - 2\witti{0}{q_{F(p)}} < \epsilon$, then, after possibly replacing $F$ with a purely transcendental extension of itself, $\anispart{(q_{F(p)})}$ contains a subform $r$ such that
\begin{enumerate} \item $\mathrm{dim}(r) = \witti{0}{q_{F(p)}}$.
\item $\mathfrak{d}_0(r) \geq \mathrm{lndeg}(p) - 1$, i.e., $r$ is divisible by a quasi-Pfister form of dimension $2^{\mathrm{lndeg}(p) -1}$. \end{enumerate}
\begin{proof} By Theorem \ref{thmsubformdec}, we may assume that there exists an anisotropic form $\tau$ of dimension $\witti{0}{q_{F(p)}}$ over $F(p)$ such that $\anispart{(p_1 \otimes \tau)} \subset \anispart{(q_{F(p)})}$. We claim that, under the given hypotheses, $\tau$ is similar to $\eta := \anispart{(p_1 \otimes \tau)}$. Exactly as in the proof of Theorem \ref{thmp1subform}, $\tau$ is certainly similar to a subform of $\eta$. It therefore suffices to show that $\mathrm{dim}(\tau) = \mathrm{dim}(\eta)$. Suppose, for the sake of contradiction, that this is not the case. Then $\mathrm{dim}(\eta) \geq \witti{0}{q_{F(p)}} + \epsilon$. Indeed, $p_1$ is, by definition, divisible by a quasi-Pfister form of dimension $2^{\mathfrak{d}_1(p)}$. By Lemma \ref{lemanisdivisibility}, the same is therefore true of $\eta$. Thus $\mathrm{dim}(\eta)$ is divisible by $2^{\mathfrak{d}_1(p)}$. Since $\mathrm{dim}(\tau) = \witti{0}{q_{F(p)}}$, the inequality $\mathrm{dim}(\eta) \geq \witti{0}{q_{F(p)}} + \epsilon$ then follows by the very definition of $\epsilon$. But, since $\eta$ is a subform of $\anispart{(q_{F(p)})}$, this yields
$$ \witti{0}{q_{F(p)}} + \epsilon \leq \mathrm{dim}(\eta) \leq \mathrm{dim}(\anispart{(q_{F(p)})}) = \mathrm{dim}(q) - \witti{0}{q_{F(p)}}, $$
which contradicts our hypothesis. We can therefore conclude that $\tau$ is similar to $\eta = \anispart{(p_1 \otimes \tau)}$. By Corollary \ref{cordivbynormform}, it follows that $\tau$ is divisible by the quasi-Pfister form $\normform{(p_1)}$. But the latter form has dimension $2^{\mathrm{lndeg}(p) - 1}$ by Lemma \ref{lemndegtower}. Thus, choosing any $a \in D(p_1)$, we see that $r := a \tau$ has the required properties. \end{proof} \end{proposition}

Proposition \ref{propquasilinearsecondmainapplication} immediately gives the following result:

\begin{theorem} \label{thmquasilinearsecondmainapplication} Let $p$ and $q$ be anisotropic quasilinear quadratic forms of dimension $\geq 2$ over $F$ and let $\epsilon$ be the unique integer in $[1,2^{\mathfrak{d}_1(p)}]$ such that $\witti{0}{q_{F(p)}} + \epsilon \equiv 0 \pmod{2^{\mathfrak{d}_1(p)}}$. Then either
\begin{enumerate} \item $\witti{0}{q_{F(p)}} \leq \frac{\mathrm{dim}(q) - \epsilon}{2}$, or
\item $\witti{0}{q_{F(p)}}$ is divisible by $2^{\mathrm{lndeg}(p) - 1}$. \end{enumerate} \end{theorem}

\begin{remark} By Corollary \ref{corp1div}, we have $2^{\mathfrak{d}_1(p)} \geq \witti{1}{p}$. Thus, the larger is $\witti{1}{p}$, the more interesting the restrictions of Theorem \ref{thmquasilinearsecondmainapplication} become. \end{remark}

\begin{example} Let $p$ and $q$ be anisotropic quasilinear quadratic forms of dimension $\geq 2$ over $F$ such that $\witti{0}{q_{F(p)}} = \frac{1}{2}\mathrm{dim}(q)$ (this is the maximal possible value of $\witti{0}{q_{F(p)}}$, and the situation is analogous to that in characteristic $\neq 2$ where one form becomes hyperbolic over the function field of the other -- see \cite[Cor. 2.32]{Scully3}). Then 
$$v_2\big(\mathrm{dim}(q)\big) \geq \mathrm{lndeg}(p) \geq \mathrm{log}_2\big(\mathrm{dim}(p)\big),$$
i.e., $\mathrm{dim}(q)$ is divisible by a power of $2$ which is at least as large as $2^{\mathrm{lndeg}(p)}$, and hence $\mathrm{dim}(p)$ (see \S \ref{secnormform}). Indeed, in this situation, we are necessarily in case (2) of Theorem \ref{thmquasilinearsecondmainapplication}. Since $\mathrm{dim}(q) = 2\witti{0}{q_{F(p)}}$, the claim follows immediately (in fact, in this case, one can show that $q$ is divisible by the quasi-Pfister form $\normform{p}$ of dimension $2^{\mathrm{lndeg}(p)}$ -- see \cite[Thm. 7.7]{Hoffmann2}). Theorem \ref{thmquasilinearsecondmainapplication} may therefore be viewed as a generalization of this result, giving us necessary conditions in order for $\witti{0}{q_{F(p)}}$ to be ``close'' to its maximal value. Again, the index being ``close'' to its maximal value is analogous to the situation in characteristic $\neq 2$ where one form becomes ``almost hyperbolic'' over the function field of another quadric. In that setting, there are no known general results in the spirit of Theorem \ref{thmquasilinearsecondmainapplication}. \end{example}

\begin{example} We can illuminate Theorem \ref{thmquasilinearsecondmainapplication} further by working through a concrete example, namely, that in which $p$ is an $(n+1)$-fold \emph{quasi-Pfister neighbour}, i.e., $\mathrm{dim}(p)>2^n$ and $p$ is similar to a subform of an anisotropic $(n+1)$-fold quasi-Pfister form (for example, $p$ could itself be a quasi-Pfister form). In this case, we have $\mathfrak{d}_1(p) = n$ and $\mathrm{lndeg}(p) = n+1$ (see \cite[Cor. 3.11]{Scully3} for details). Theorem \ref{thmquasilinearsecondmainapplication} therefore implies that if $q$ is an anisotropic quasilinear quadratic form of dimension $\geq 2$ over $F$, then either 
\begin{enumerate} \item $\witti{0}{q_{F(p)}} \leq \frac{\mathrm{dim}(q) - \epsilon}{2}$, or
\item $\witti{0}{q_{F(p)}}$ is divisible by $2^n$, \end{enumerate}
where $\epsilon$ is the unique integer in $[1,2^n]$ such that $\witti{0}{q_{F(p)}} + \epsilon \equiv 0 \pmod{2^n}$. Over fields of characteristic $\neq 2$, the same statement is known to hold in the (very special) case where $\witti{0}{q_{F(p)}} = \frac{1}{2}\mathrm{dim}(q)$; indeed, a classical result due to Arason and Pfister asserts that a quadratic form which becomes hyperbolic over the function field of a Pfister quadric is necessarily divisible by the corresponding Pfister form (see \cite[Cor. 23.6]{EKM}). However, there seems to be no known generalization of this result similar to the one just given for quasilinear forms. \end{example}

\noindent {\bf Acknowledgements.} The author gratefully acknowledges the support of a PIMS postdoctoral fellowship and NSERC discovery grant during the preparation of this article.

\bibliographystyle{alphaurl}

\end{document}